\documentclass[a4paper,reqno]{amsart}
\usepackage{latexsym,amsmath,amssymb,amsfonts}
\usepackage{a4wide}
\usepackage{color}
\usepackage[bookmarks]{hyperref} 

\theoremstyle{plain}
\begingroup
\newtheorem{theorem}{Theorem}[section]
\newtheorem{lemma}[theorem]{Lemma}
\newtheorem{proposition}[theorem]{Proposition}
\newtheorem{corollary}[theorem]{Corollary}
\endgroup
\theoremstyle{definition}
\begingroup
\newtheorem{definition}[theorem]{Definition}
\newtheorem{remark}[theorem]{Remark}

\endgroup
\theoremstyle{remark}

\mathchardef\emptyset="001F

\numberwithin{equation}{section}

\newcommand{\e}{\varepsilon}
\newcommand{\vphi}{\varphi}
\newcommand{\N}{\mathbb N}

\newcommand{\R}{\mathbb R}

\renewcommand{\div}{{\rm div}}

\newcommand{\LL}{{\mathcal L}}
\newcommand{\hn}{{\mathcal H}^{N-1}}
\newcommand{\wto}{\rightharpoonup}
\newcommand{\setmeno}{\!\setminus\!}
\newcommand{\dd}{\mathrm{d}}

\newcommand{\f}{\mathcal{F}}
\newcommand{\p}{\mathcal{P}}
\newcommand{\nl}{\mathcal{NL}}
\newcommand{\mloc}{m_{\rm loc}}
\newcommand{\mglob}{m_{\rm glob}}
\newcommand{\I}{\mathcal{I}}
\newcommand{\ftil}{\widetilde{F}_h}

\newcommand{\average}{{\mathchoice {\kern1ex\vcenter{\hrule height.4pt
width 6pt
depth0pt} \kern-9.7pt} {\kern1ex\vcenter{\hrule height.4pt width 4.3pt
depth0pt}
\kern-7pt} {} {} }}
\newcommand{\med}{\average\int}

\newdimen\mex
\makeatletter
\def\niv{\mathrel{\hbox{\hglue -0.4\mex
\vrule \@height 1\mex \@width .1\mex
\vrule \@height .1\mex \@width 1\mex
\hglue -0.2\mex}}}
\makeatother

\title[Local and global minimality results for a nonlocal isoperimetric problem on $\R^N$]{Local and global minimality results for a nonlocal isoperimetric problem on $\R^N$}
\author{M.\ Bonacini}
\author{R.\ Cristoferi}
\address[M.~Bonacini]{SISSA, Via Bonomea 265, 34136 Trieste, Italy}
\email[Marco Bonacini]{marco.bonacini@sissa.it}
\address[R.\ Cristoferi]{SISSA, Via Bonomea 265, 34136 Trieste, Italy}
\email[Riccardo Cristoferi]{riccardo.cristoferi@sissa.it}

\thanks{Preprint SISSA 32/2013/MATE}

\begin{document}

\begin{abstract}
We consider a nonlocal isoperimetric problem defined in the whole space $\R^N$, whose nonlocal part is given by a Riesz potential with exponent $\alpha\in(0,N - 1)$. We show that
critical configurations with positive second variation are local minimizers and satisfy a quantitative inequality with respect to the $L^1$-norm. This criterion provides the existence
of a (explicitly determined) critical threshold determining the interval of volumes for which the ball is a local minimizer, and allows to address several global minimality issues.
\end{abstract}

\maketitle


\section{Introduction} \label{sect:intro}

In this paper we provide a description of the energy landscape of the family of functionals
\begin{equation} \label{eq:functional}
\f(E) := \p(E) + \gamma\int_{\R^N}\int_{\R^N} \frac{\chi_E(x)\chi_E(y)}{|x-y|^\alpha}\,\dd x\dd y\,,
\qquad\alpha\in(0,N-1),\,\gamma>0
\end{equation}
defined over finite perimeter sets $E\subset\R^N$, $N\geq2$.
Here $\p(E)$ denotes the perimeter of $E$ in $\R^N$ and $\chi_E$ its characteristic function.
In the following we will usually denote the second term in the functional by $\nl_{\alpha}(E)$,
and we will omit the subscript $\alpha$ when there is no risk of ambiguity.

The nature of the energy \eqref{eq:functional} appears in the modeling of different physical problems.
The most physically relevant case is in three dimensions with $\alpha=1$,
where the nonlocal term corresponds to a Coulombic repulsive interaction:
one of the first examples is the celebrated Gamow's water-drop model for the constitution of the atomic nucleus (see \cite{Gamow}), and energies of this kind are also related
(via $\Gamma$-convergence) to the Ohta-Kawasaki model for diblock copolymers (see \cite{OhtaKaw}).
For a more specific account on the physical background of this kind of problems, we refer to \cite{Mur}.

From a mathematical point of view, functionals of the form \eqref{eq:functional} recently drew the attention of many authors
(see \cite{AceFusMor,ChoPel1,ChoPel2,CicSpa,GolMurSerI,GolMurSerII,Jul,KnuMur1,KnuMur2,LuOtt,Mur2,SteTop}).
The main feature of the energy \eqref{eq:functional} is the presence of two competing terms,
the sharp interface energy and the long-range repulsive interaction.
Indeed, while the first term is minimized by the ball (by the isoperimetric inequality),
the nonlocal term is in fact maximized by the ball,
as a consequence of the Riesz's rearrangement inequality (see \cite[Theorem~3.7]{LiebLoss}),
and favours scattered or oscillating configurations.
Hence, due to the presence of this competition in the structure of the problem, the minimization of $\f$ is highly non trivial.

We remark that, by scaling, minimizing \eqref{eq:functional} under the volume constraint $|E|=m$
is equivalent to the minimization of the functional
$$
\p(E) + \gamma \left(\frac{m}{|B_1|}\right)^{\frac{N-\alpha+1}{N}} \nl_\alpha(E)
$$
under the constraint $|E|=|B_1|$, where $B_1$ is the unit ball in $R^N$.
It is clear from this expression that, for small masses, the perimeter is the leading term and this suggests that in this case the ball should be the solution to the minimization
problem; on the other hand, for large masses the nonlocal term becomes dominant and causes the existence of a solution to fail.

Indeed it was proved, although not in full generality, that the functional $\f$ is uniquely minimized (up to translations) by the ball
for every value of the volume below a critical threshold:
in \cite{KnuMur1} for the planar case $N=2$, in \cite{KnuMur2} for $3\leq N\leq7$,
and in \cite{Jul} for any dimension $N$ but for $\alpha=N-2$.
Moreover, the existence of a critical mass above which the minimum problem does not admit a solution was established
in \cite{KnuMur1} in dimension $N=2$, in \cite{KnuMur2} for every $N$ and for $\alpha\in(0,2)$,
and in \cite{LuOtt} in the physical interesting case $N=3$, $\alpha=1$.

In this paper we aim at providing a more detailed picture of the energy landscape of the functional \eqref{eq:functional}
by a totally different approach, based on the positivity of the second variation and mainly inspired by \cite{AceFusMor}
(which deals with the same energy functional, with $\alpha=N-2$, in a periodic setting).
First, we reobtain and strengthen some of the above results, proving in full generality
that the ball is the unique global minimizer for small masses,
without restrictions on the parameters $N$ and $\alpha$ (Theorem~\ref{teo:minglobpalla}).

Moreover, for $\alpha$ small we also provide a complete characterization of the ground state,
showing that the ball is the unique global minimizer, as long as a minimizer exists (Theorem~\ref{teo:apiccolo}),
and that in this regime we can write $(0,\infty)=\cup_k(m_k,m_{k+1}]$, with $m_{k+1}>m_k$, in such a way that for $m\in[m_{k-1},m_{k}]$
a minimizing sequence for the functional is given by a configuration of at most $k$ disjoint balls with diverging mutual distance (Theorem~\ref{teo:apiccolo2}).
The results stated in Theorem~\ref{teo:apiccolo} and Theorem~\ref{teo:apiccolo2} are completely new (except for the first one, which was proved only in the special case $N=2$ in \cite{KnuMur1}).

Finally, we also investigate for the first time in this context the issue of \emph{local minimizers}, that is, sets which minimize the energy with respect to competitors sufficiently close in the $L^1$-sense (where we
measure the distance between two sets by the quantity \eqref{eq:asimmetria}, which takes into account the translation invariance of the functional).
For any $N$ and $\alpha$ we show the existence of a volume threshold below which the ball is also an isolated local minimizer, determining it explicitly in the three dimensional case with a Newtonian
potential (Theorem~\ref{teo:minlocpalla}).

As anticipated above, our methods follow a second variation approach which has been recently developed and applied to different variational problems, whose common feature is the fact
that the energy functionals are characterized by the competition between bulk energies and surface energies (see \cite{FusMor} in the context of epitaxially strained elastic films,
\cite{CagMorMor, BonMor} for the Mumford-Shah functional, \cite{CapJulPis} for a variational model for cavities in elastic bodies).
In particular we stress the attention on \cite{AceFusMor}, which deals with energies in the form \eqref{eq:functional} in a periodic setting (see also \cite{JulPis}, where the same problem is considered in an open set with Neumann boundary conditions).
The basic idea is to associate with the second variation of $\f$ at a regular critical set $E$, a quadratic form defined on the functions $\vphi\in H^1(\partial E)$ such that
$\int_{\partial E}\vphi=0$, whose non-negativity is easily seen to be a necessary condition for local minimality.

In fact, one of the main tools in proving the aforementioned results is represented by Theorem~\ref{teo:minl1}, where we show that the strict
positivity of this quadratic form is sufficient for local minimality with respect to $L^1$-perturbations.
Although the general strategy to establish this theorem follows the one developed in \cite{FusMor,AceFusMor},
we have to tackle the nontrivial technical difficulties coming from working with a more general nonlocal term (the exponent $\alpha$ is allowed to range in the whole interval $(0,N-1)$) and from the lack of compactness of the ambient space $\R^N$.
The proof passes through an intermediate step, which amounts to showing that a critical set $E$ with positive second variation is a minimizer with respect to sets whose boundary is a
normal graph on $\partial E$ with $W^{2,p}$-norm sufficiently small (Theorem~\ref{teo:minw}).
Then a contradiction argument,
which is mainly based on the regularity properties of quasi-minimizers of the area functional
(and which is close in spirit to the ideas introduced by Cicalese and Leonardi \cite{CicLeo} to provide an alternative proof of the standard quantitative isoperimetric inequality)
leads to the conclusion that this weaker notion of local minimality implies the local minimality in the $L^1$-sense.

The proofs of the global minimality results described above require additional arguments and nontrivial refinements of the previous ideas.

An issue which remains unsolved is concerned with the structure of the set of masses for which the problem does not have a solution: is it always true that it has the form
$(m,+\infty)$ for all the values of $\alpha$ and $N$? Notice that we provide a positive answer to this question in the case where $\alpha$ is small.
Another interesting question asks if there are other global (or local) minimizers different from the ball.
Finally, our analysis leaves open the case of $\alpha\in[N-1,N)$, which seems to require different techniques.

\smallskip
This paper is organized as follows.
In Section~\ref{sect:setting} we set up the problem and we list the main results of this work.
The notion of second variation of the functional $\f$ is introduced in Section~\ref{sect:2var},
together with the first part of the proof of the main result of the paper, which is completed in Section~\ref{sect:l1min}.
In Section~\ref{sect:ball} we compute explicitly the second variation of the ball, and we discuss its local minimality by applying our sufficiency criterion. Finally, Section~\ref{sect:global} is devoted to the proof of the results concerning the global minimality issues.


\section{Statements of the results} \label{sect:setting}

We start our analysis with some preliminary observations about the features of the energy functional \eqref{eq:functional},
before listing the main results of this work.
For a finite perimeter set $E$, we will denote by $\nu_E$ the exterior generalized unit normal to $\partial^*E$,
and we will not indicate the dependence on the set $E$ when no confusion is possible.

Given a measurable set $E\subset\R^N$, we introduce an auxiliary function $v_E$ by setting
\begin{equation} \label{eq:ve}
v_E(x) := \int_E \frac{1}{|x-y|^\alpha}\,\dd y
\qquad\text{for }x\in\R^N.
\end{equation}
The function $v_E$ can be characterized as the solution to the equation
\begin{equation} \label{eq:ve2}
(-\Delta)^s v_E = c_{N,s}\,\chi_E\,,
\qquad s=\frac{N-\alpha}{2}
\end{equation}
where $(-\Delta)^s$ denotes the fractional laplacian and $c_{N,s}$ is a constant depending on the dimension and on $s$
(see \cite{PalVal} for an introductory account on this operator and the references contained therein).
Notice that we are interested in those values of $s$ which range in the interval $(\frac12,\frac{N}{2})$.
We collect in the following proposition some regularity properties of the function $v_E$.

\begin{proposition} \label{prop:ve}
Let $E\subset\R^N$ be a measurable set with $|E|\leq m$.
Then there exists a constant $C$, depending only on $N,$ $\alpha$ and $m$, such that
\begin{equation*}
\|v_E\|_{W^{1,\infty}(\R^N)}\leq C\,.
\end{equation*}
Moreover, $v_E\in C^{1,\beta}(\R^N)$ for every $\beta<N-\alpha-1$ and
\begin{equation*}
\|v_E\|_{C^{1,\beta}(\R^N)}\leq C'
\end{equation*}
for some positive constant $C'$ depending only on $N,$ $\alpha,$ $m$ and $\beta$.
\end{proposition}

\begin{proof}
The first part of the result is proved in \cite[Lemma~4.4]{KnuMur2}, but we repeat here the easy proof for the reader's convenience.
By \eqref{eq:ve},
\begin{align*}
v_E(x)
= \int_{B_1(x)\cap E}\frac{1}{|x-y|^{\alpha}}\,\dd y + \int_{E\setminus B_1(x)}\frac{1}{|x-y|^\alpha}\,\dd y
\leq \int_{B_1}\frac{1}{|y|^\alpha}\,\dd y + m \leq C.
\end{align*}
By differentiating \eqref{eq:ve} in $x$ and arguing similarly, we obtain
\begin{align*}
|\nabla v_E(x)|
\leq \alpha \int_{E}\frac{1}{|x-y|^{\alpha+1}}\,\dd y
\leq \alpha \int_{B_1}\frac{1}{|y|^{\alpha+1}}\,\dd y + \alpha\,m \leq C.
\end{align*}
Finally, by adding and subtracting the term $\frac{(x-y)|z-y|^\beta}{|x-y|^{\alpha+\beta+2}} - \frac{(z-y)|x-y|^\beta}{|z-y|^{\alpha+\beta+2}}$, we can write
\begin{align} \label{stimeholder}
|\nabla v_E(x) - \nabla v_E(z)|
& \leq \alpha \int_E \bigg| \frac{x-y}{|x-y|^{\alpha+2}}-\frac{z-y}{|z-y|^{\alpha+2}}\bigg| \,dy \nonumber\\
&\leq \alpha\int_E \biggl(\frac{1}{|x-y|^{\alpha+\beta+1}}+\frac{1}{|z-y|^{\alpha+\beta+1}}\biggr) \big||x-y|^\beta-|z-y|^\beta\big|\,\dd y \\
&\quad+\alpha\int_E \bigg| \frac{(x-y)|z-y|^\beta}{|x-y|^{\alpha+\beta+2}} - \frac{(z-y)|x-y|^\beta}{|z-y|^{\alpha+\beta+2}} \bigg|\,\dd y \nonumber
\end{align}
Observe now that for every $v,w\in\R^N\setminus\{0\}$
\begin{align*}
\bigg| \frac{v}{|v|}|w|^{\alpha+2\beta+1} - \frac{w}{|w|}|v|^{\alpha+2\beta+1} \bigg|
&= \big| v\,|v|^{\alpha+2\beta} - w\,|w|^{\alpha+2\beta} \big|
\leq C\max\{|v|,|w|\}^{\alpha+2\beta}|v-w|\\
&\leq C\max\{|v|,|w|\}^{\alpha+\beta+1}|v-w|^{\beta}
\end{align*}
where $C$ depends on $N,$ $\alpha$ and $\beta$.
Using this inequality to estimate the second term in \eqref{stimeholder} we deduce
\begin{align*}
|\nabla v_E(x) - \nabla v_E(z)|
&\leq \alpha|x-z|^\beta\int_E \biggl(\frac{1}{|x-y|^{\alpha+\beta+1}}+\frac{1}{|z-y|^{\alpha+\beta+1}} + \frac{C}{\min\{|x-y|,|z-y|\}^{\alpha+\beta+1}}\biggr) \,\dd y
\end{align*}
which completes the proof of the proposition, since the last integral is bounded by a constant
depending only on $N,$ $\alpha,$ $m$ and $\beta$.
\end{proof}

\begin{remark} \label{rm:ve}
In the case $\alpha=N-2$, the function $v_E$ solves the equation $-\Delta v_E = c_N\,\chi_E$, and the nonlocal term is exactly
$$
\nl_{N-2}(E) = \int_{\R^N}|\nabla v_E(x)|^2\,\dd x\,.
$$
By standard elliptic regularity, $v_E\in W^{2,p}_{\rm loc}(\R^N)$ for every $p\in[1,+\infty)$.
\end{remark}

The following proposition contains an auxiliary result which will be used frequently in the rest of the paper.

\begin{proposition}[Lipschitzianity of the nonlocal term] \label{prop:NLlipschitz}
Given $\bar{\alpha}\in(0,N-1)$ and $m\in(0,+\infty)$,
there exists a constant $c_0$, depending only on $N,$ $\bar{\alpha}$ and $m$ such that
if $E,F\subset\R^N$ are measurable sets with $|E|,|F|\leq m$ then
\begin{equation*}
|\nl_{\alpha}(E)-\nl_{\alpha}(F)| \leq c_0|E\triangle F|
\end{equation*}
for every $\alpha\leq\bar{\alpha}$,
where $\triangle$ denotes the symmetric difference of two sets.
\end{proposition}

\begin{proof}
We have that
\begin{align*}
\nl_\alpha(E) - \nl_\alpha(F)
&= \int_{\R^N}\int_{\R^N} \left( \frac{\chi_E(x)(\chi_E(y)-\chi_F(y))}{|x-y|^\alpha}
 + \frac{\chi_F(y)(\chi_E(x)-\chi_F(x))}{|x-y|^\alpha} \right) \, \dd x \dd y \\
&= \int_{E\setminus F} \bigl(v_E(x)+v_F(x)\bigr)\,\dd x - \int_{F\setminus E} \bigl(v_E(x)+v_F(x)\bigr)\,\dd x \\
& \leq \int_{E\triangle F} \bigl(v_E(x)+v_F(x)\bigr)\,\dd x
 \leq 2C\,|E\triangle F|,
\end{align*}
where the constant $C$ is provided by Proposition~\ref{prop:ve},
whose proof shows also that it can be chosen independently of $\alpha\leq\bar{\alpha}$.
\end{proof}

The issue of existence and characterization of \emph{global minimizers} of the problem
\begin{equation} \label{eq:minimumprobl}
\min\left\{\f(E) \,:\, E\subset\R^N, \, |E|=m\right\}\,,
\end{equation}
for $m>0$, is not at all an easy task.
A principal source of difficulty in applying the direct method of the Calculus of Variations
comes from the lack of compactness of the space with respect to $L^1$ convergence of sets
(with respect to which the functional is lower semicontinuous).
It is in fact well known that the minimum problem \eqref{eq:minimumprobl}
does not admit a solution for certain ranges of masses.

Besides the notion of global minimality, we will address also the study of sets which minimize locally the functional
with respect to small $L^1$-perturbations.
By translation invariance, we measure the $L^1$-distance of two sets modulo translations by the quantity
\begin{equation} \label{eq:asimmetria}
\alpha(E,F) := \min_{x\in\R^N} |E\triangle(x+F)|\,.
\end{equation}

\begin{definition} \label{def:minloc}
We say that $E\subset\R^N$ is a \emph{local minimizer} for the functional \eqref{eq:functional}
if there exists $\delta>0$ such that
$$
\f(E)\leq\f(F)
$$
for every $F\subset\R^N$ such that $|F|=|E|$ and $\alpha(E,F)\leq\delta$.
We say that $E$ is an \emph{isolated local minimizer} if the previous inequality is strict whenever $\alpha(E,F)>0$.
\end{definition}

The first order condition for minimality, coming from the first variation of the functional
(see \eqref{eq:firvar}, and also \cite[Theorem~2.3]{ChoSte}),
requires a $C^2$-minimizer $E$ (local or global) to satisfy the Euler-Lagrange equation
\begin{equation} \label{eq:EL}
H_{\partial E}(x) + 2\gamma v_E(x) = \lambda
\qquad\text{for every }x\in\partial E
\end{equation}
for some constant $\lambda$ which plays the role of a Lagrange multiplier associated with the volume constraint.
Here $H_{\partial E}:= \div_{\tau}\nu_E(x)$ denotes the sum of the principal curvatures of $\partial E$
($\div_{\tau}$ is the tangential divergence on $\partial E$, see \cite[Section~7.3]{AFP}).
Following \cite{AceFusMor}, we define \emph{critical sets} as those satisfying \eqref{eq:EL} in a weak sense,
for which further regularity can be gained \emph{a posteriori} (see Remark~\ref{rm:regularitycrit}).

\begin{definition} \label{def:crit}
We say that $E\subset\R^N$ is a \emph{regular critical set} for the functional \eqref{eq:functional}
if $E$ is a bounded set of class $C^1$ and \eqref{eq:EL} holds weakly on $\partial E$, \emph{i.e.},
\begin{equation*}
\int_{\partial E}\div_{\tau}\zeta\,\dd \hn = -2\gamma\int_{\partial E}v_E\,\langle\zeta,\nu_E\rangle\,\dd \hn
\end{equation*}
for every $\zeta\in C^{1}(\R^N;\R^N)$ such that $\int_{\partial E}\langle\zeta,\nu_E\rangle\,\dd\hn=0$.
\end{definition}

\begin{remark} \label{rm:regularitycrit}
By Proposition~\ref{prop:ve} and by standard regularity (see, \emph{e.g.}, \cite[Proposition~7.56 and Theorem~7.57]{AFP})
a critical set $E$ is of class $W^{2,2}$ and $C^{1,\beta}$ for all $\beta\in(0,1)$.
In turn, recalling Proposition~\ref{prop:ve},
by Schauder estimates (see \cite[Theorem~9.19]{GT})
we have that $E$ is of class $C^{3,\beta}$ for all $\beta\in(0,N-\alpha-1)$.
\end{remark}

We collect in the following theorem some regularity properties of local and global minimizers,
which are mostly known (see, for instance, \cite{KnuMur2,LuOtt,SteTop} for global minimizers,
and \cite{AceFusMor} for local minimizers in a periodic setting).
The basic idea is to show that a minimizer solves a suitable penalized minimum problem,
where the volume constraint is replaced by a penalization term in the functional,
and to deduce that a quasi-minimality property is satisfied (see Definition~\ref{def:quamin}).

\begin{theorem} \label{teo:regmin}
Let $E\subset\R^N$ be a global or local minimizer for the functional \eqref{eq:functional} with volume $|E|=m$.
Then the reduced boundary $\partial^*E$ is a $C^{3,\beta}$-manifold for all $\beta<N-\alpha-1$,
and the Hausdorff dimension of the singular set satisfies $\dim_{\mathcal{H}}(\partial E \setminus \partial^*E)\leq N-8$.
Moreover, $E$ is (essentially) bounded.
Finally, every global minimizer is connected, and every local minimizer has at most a finite number of connected components
\footnote{Here and in the rest of the paper \emph{connectedness} is intended in a measure-theoretic sense:
a \emph{connected component} of $E$ is defined as any subset $F\subset E$ with $|F|>0$, $|E\setminus F|>0$, such that $\p(E)=\p(F)+\p(E\setminus F)$.}.
\end{theorem}

\begin{proof}
We divide the proof into three steps, following the ideas contained in \cite[Proposition~2.7 and Theorem~2.8]{AceFusMor}
in the first part.

\smallskip
\noindent
{\it Step 1.}
We claim that there exists $\Lambda>0$ such that $E$ is a solution to the penalized minimum problem
\begin{equation*}
\min\left\{\f(F) + \Lambda \big||F|-|E|\big| \,:\, F\subset\R^N, \, \alpha(F,E)\leq\frac{\delta}{2}\right\}\,,
\end{equation*}
where $\delta$ is as in Definition~\ref{def:minloc}
(the obstacle $\alpha(F,E)\leq\frac{\delta}{2}$ is not present in the case of a global minimizer).
To obtain this, it is in fact sufficient to show that there exists $\Lambda>0$ such that
if $F\subset\R^N$ satisfies $\alpha(F,E)\leq\frac{\delta}{2}$ and $\f(F)+\Lambda\big||F|-|E|\big|\leq\f(E)$,
then $|F|=|E|$.

Assume by contradiction that there exist sequences $\Lambda_h\to+\infty$ and $E_h\subset\R^N$
such that $\alpha(E_h,E)\leq\frac{\delta}{2}$, $\f(E_h)+\Lambda_h\big||E_h|-|E|\big|\leq\f(E)$, and $|E_h|\neq|E|$.
Notice that, since $\Lambda_h\to+\infty$, we have $|E_h|\to|E|$.

We define new sets $F_h:=\lambda_hE_h$, where $\lambda_h=\left(\frac{|E|}{|E_h|}\right)^{\frac1N}\to1$, so that $|F_h|=|E|$.
Then we have, for $h$ sufficiently large, that $\alpha(F_h,E)\leq\delta$ and
\begin{align*}
\f(F_h)
& = \f(E_h) 
  + (\lambda_h^{N-1}-1)\p(E_h) + \gamma(\lambda_h^{2N-\alpha}-1)\nl_\alpha(E_h)\\
& \leq \f(E)
  + (\lambda_h^{N-1}-1)\p(E_h) + \gamma(\lambda_h^{2N-\alpha}-1)\nl_\alpha(E_h) -\Lambda_h \big||E_h|-|E|\big| \\
& = \f(E) + |\lambda_h^N-1|\,|E_h|
  \left( \frac{\lambda_h^{N-1}-1}{|\lambda_h^N-1|}\,\frac{\p(E_h)}{|E_h|}
  + \gamma\,\frac{\lambda_h^{2N-\alpha}-1}{|\lambda_h^N-1|}\,\frac{\nl_\alpha(E_h)}{|E_h|} - \Lambda_h \right)
  < \f(E)\,,
\end{align*}
which contradicts the local minimality of $E$
(notice that the same proof works also in the case of global minimizers).

\smallskip
\noindent
{\it Step 2.}
From the previous step, it follows that $E$ is an $(\omega,r_0)$-minimizer for the area functional
for suitable $\omega>0$ and $r_0>0$ (see Definition~\ref{def:quamin}).
Indeed, choose $r_0$ such that $\omega_Nr_0^N\leq\frac{\delta}{2}$:
then if $F$ is such that $F\triangle E\subset\subset B_r(x)$ with $r<r_0$,
we clearly have that $\alpha(F,E)\leq\frac{\delta}{2}$
and by minimality of $E$ we deduce that
\begin{align*}
\p(E)
&\leq \p(F) + \gamma \bigl( \nl_\alpha(F)-\nl_\alpha(E)\bigr) + \Lambda \big||F|-|E|\big|\\
&\leq \p(F) + \bigl( \gamma c_0 + \Lambda \bigr) |E\triangle F|
\end{align*}
(using Proposition~\ref{prop:NLlipschitz}),
and the claim follows with $\omega:=\gamma c_0 + \Lambda$.

\smallskip
\noindent
{\it Step 3.}
The $C^{1,\frac12}$-regularity of $\partial^*E$, as well as the condition on the Hausdorff dimension of the singular set,
follows from classical regularity results for $(\omega,r_0)$-minimizers (see, \emph{e.g.}, \cite[Theorem~1]{Tam}).
In turn, the $C^{3,\beta}$-regularity follows from the Euler-Lagrange equation, as in Remark~\ref{rm:regularitycrit}.

To show the essential boundedness, we use the density estimates for $(\omega,r_0)$-minimizers of the perimeter,
which guarantee the existence of a positive constant $\vartheta_0>0$ (depending only on $N$) such that
for every point $y\in\partial^{*} E$ and $r<\min\{r_0,1/(2N\omega)\}$
\begin{equation} \label{eq:densitylb}
\p(E;B_r(y))\geq\vartheta_0r^{N-1}
\end{equation}
(see, \emph{e.g.}, \cite[Theorem~21.11]{Mag}).
Assume by contradiction that there exists a sequence of points $x_n\in\R^N\setminus E^{(0)}$, where
$$
E^{(0)} := \left\{x\in\R^N : \limsup_{r\to0^+}\frac{|E\cap B_r(x)|}{r^N}=0\right\}\,,
$$
such that $|x_n|\to+\infty$. Fix $r<\min\{r_0,1/(2N\omega)\}$ and assume without loss of generality that $|x_n-x_m|>4r$.
It is easily seen that for infinitely many $n$ we can find $y_n\in\partial^*E\cap B_r(x_n)$; then
\begin{align*}
\p(E) \geq \sum_n\p(E,B_r(y_n)) \geq \sum_n \vartheta_0 r^{N-1} =+\infty\,,
\end{align*}
which is a contradiction.

Connectedness of global minimizers follows easily from their boundedness, since if a global minimizer had at least two connected components
one could move one of them far apart from the others without changing the perimeter but decreasing the nonlocal term in the energy
(see \cite[Lemma~3]{LuOtt} for a formal argument).

Finally, let $E_0$ be a connected component of a local minimizer $E$: then,
denoting by $B_r$ a ball with volume $|B_r|=|E_0|$,
using the isoperimetric inequality and the fact that $E$ is a $(\omega,r_0)$-minimizer for the area functional, we obtain
\begin{align*}
\p(E\setminus E_0) + N\omega_Nr^{N-1}
&\leq \p(E\setminus E_0) + \p(E_0)
= \p(E) \\
&\leq \p(E\setminus E_0) +\omega|E_0|
= \p(E\setminus E_0) + \omega\omega_Nr^N,
\end{align*}
which is a contradiction if $r$ is small enough.
This shows an uniform lower bound on the volume of each connected component of $E$,
from which we deduce that $E$ can have at most a finite number of connected components.
\end{proof}

We are now ready to state the main results of this paper.
The central theorem, whose proof lasts for Sections~\ref{sect:2var} and \ref{sect:l1min},
provides a sufficiency local minimality criterion based on the second variation of the functional.
Following \cite{AceFusMor} (see also \cite{ChoSte}), we introduce a quadratic form
associated with the second variation of the functional at a regular critical set (see Definition~\ref{def:fquad});
then we show that its strict positivity
(on the orthogonal complement to a suitable finite dimensional subspace of directions where the second variation degenerates,
due to translation invariance) is a sufficient condition for isolated local minimality, according to Definition~\ref{def:minloc},
by proving a \emph{quantitative} stability inequality.
The result reads as follows.

\begin{theorem} \label{teo:minl1}
Assume that $E$ is a regular critical set for $\f$ with positive second variation, in the sense of Definition~\ref{def:varpos}.
Then there exist $\delta>0$ and $C>0$ such that
\begin{equation}\label{eq:quant}
\f(F)\geq\f(E)+C\bigl(\alpha(E,F)\bigr)^2
\end{equation}
for every $F\subset\R^N$ such that $|F|=|E|$ and $\alpha(E,F)<\delta$.
\end{theorem}

The local minimality criterion in Theorem~\ref{teo:minl1} can be applied
to obtain information about local and global minimizers of the functional \eqref{eq:functional}.
In order to state the results more clearly, we will underline the dependence of the functional on the parameters $\alpha$ and $\gamma$
by writing $\f_{\alpha,\gamma}$ instead of $\f$.
We start with the following theorem, which shows the existence of a critical mass $\mloc$ such that
the ball $B_R$ is an isolated local minimizer if $|B_R|<\mloc$,
but is no longer a local minimizer for larger masses.
We also determine explicitly the volume threshold in the three-dimensional case.
The result, which to the best of our knowledge provides the first characterization of the local minimality of the ball,
will be proved in Section~\ref{sect:ball}.

\begin{theorem}[Local minimality of the ball] \label{teo:minlocpalla}
Given $N\geq 2$, $\alpha\in(0,N-1)$ and $\gamma>0$,
there exists a critical threshold $\mloc=\mloc(N,\alpha,\gamma)>0$ such that
the ball $B_R$ is an isolated local minimizer for $\f_{\alpha,\gamma}$, in the sense of Definition~\ref{def:minloc},
if $0<|B_R|<\mloc$.

If $|B_R|>\mloc$, there exists $E\subset\R^N$ with $|E|=|B_R|$ and $\alpha(E,B_R)$ arbitrarily small
such that $\f_{\alpha,\gamma}(E)<\f_{\alpha,\gamma}(B_R)$.

In particular, in dimension $N=3$ we have
$$
\mloc(3,\alpha,\gamma)=\frac43\,\pi\Biggl(\frac{(6-\alpha)(4-\alpha)}{2^{3-\alpha}\gamma\alpha\pi}\Biggr)^{\frac{3}{4-\alpha}}.
$$

Finally $\mloc(N,\alpha,\gamma)\to\infty$ as $\alpha\to0^+$.
\end{theorem}

Our local minimality criterion allows us to deduce further properties about global minimizers,
which will be proved in Section~\ref{sect:global}.
The first result states that the ball is the unique global minimizer of the functional for small masses.
We provide an alternative proof of this fact
(which was already known in the literature in some particular cases, as explained in the introduction)
which holds in full generality and removes the restrictions on the parameters $N$ and $\alpha$ which were present in the previous partial results.

\begin{theorem}[Global minimality of the ball] \label{teo:minglobpalla}
Let $\mglob(N,\alpha,\gamma)$ be the supremum of the masses $m>0$ such that the ball of volume $m$ is a global minimizer of $\f_{\alpha,\gamma}$ in $\R^N$.
Then $\mglob(N,\alpha,\gamma)$ is positive and finite,
and the ball of volume $m$ is a global minimizer of $\f_{\alpha,\gamma}$ if $m\leq\mglob(N,\alpha,\gamma)$.
Moreover, it is the unique (up to translations) global minimizer of $\f_{\alpha,\gamma}$ if $m<\mglob(N,\alpha,\gamma)$.
\end{theorem}

In the following theorems we analyze the global minimality issue for $\alpha$ close to 0, showing in particular that in this case the unique minimizer, as long as a minimizer exists,
is the ball, and characterizing the infimum of the energy when the problem does not have a solution.

\begin{theorem}[Characterization of global minimizers for $\alpha$ small] \label{teo:apiccolo}
There exists $\bar{\alpha}=\bar{\alpha}(N,\gamma)>0$ such that for every $\alpha<\bar{\alpha}$
the ball with volume $m$ is the unique (up to translations) global minimizer of $\f_{\alpha,\gamma}$ if $m\leq\mglob(N,\alpha,\gamma)$,
while for $m>\mglob(N,\alpha,\gamma)$ the minimum problem for $\f_{\alpha,\gamma}$ does not have a solution.
\end{theorem}

\begin{theorem}[Characterization of minimizing sequences for $\alpha$ small] \label{teo:apiccolo2}
Let $\alpha<\bar{\alpha}$ (where $\bar{\alpha}$ is given by Theorem~\ref{teo:apiccolo}) and let
$$
f_k(m) := \min_{\stackrel{\mu_1,\ldots,\mu_k\geq0}{\mu_1+\ldots+\mu_k=m}} \biggl\{\sum_{j=1}^k \f(B^i) : B^i\text{ \rm ball, }|B^i|=\mu_i \biggr\}.
$$
There exists an increasing sequence $(m_k)_k$, with $m_0=0$, $m_1=\mglob$, such that $\lim_{k}m_k=\infty$ and
\begin{equation} \label{eq:apiccolo2}
\inf_{|E|=m}\f(E) = f_k(m) \qquad\text{for every }m\in[m_{k-1},m_k], \text{ for all } k\in\N,
\end{equation}
that is, for every $m\in[m_{k-1},m_k]$ a minimizing sequence for the total energy is obtained by a configuration of at most $k$ disjoint balls with diverging mutual distance.
Moreover, the number of non-degenerate balls tends to $+\infty$ as $m\to+\infty$.
\end{theorem}

\begin{remark}
Since $\mloc(N,\alpha,\gamma)\to+\infty$ as $\alpha\to0^+$
and the non-existence threshold is uniformly bounded for $\alpha\in(0,1)$
(see Proposition~\ref{prop:nonexist}),
we immediately deduce that, for $\alpha$ small, $\mglob(N,\alpha,\gamma)<\mloc(N,\alpha,\gamma)$.
Moreover, by comparing the energy of a ball of volume $m$ with the energy of two disjoint balls of volume $\frac{m}{2}$,
and sending to infinity the distance between the balls,
we deduce after a straightforward computation
(and estimating $\nl_\alpha(B_1)\geq \omega_N^2 2^{-\alpha}$)
that the following upper bound for the global minimality threshold of the ball holds:
$$
\mglob(N,\alpha,\gamma) <
\omega_N \biggl( \frac{2^\alpha N (2^{\frac1N}-1)}{\omega_N\gamma (1-(\frac12)^{\frac{N-\alpha}{N}})} \biggr)^{\frac{N}{N+1-\alpha}}\,.
$$
Hence, by comparing this value with the explicit expression of $\mloc$ in the physical interesting case $N=3$, $\alpha=1$
(see Theorem~\ref{teo:minlocpalla}),
we deduce that $\mglob(3,1,\gamma)<\mloc(3,1,\gamma)$.
\end{remark}

\begin{remark} \label{rm:log}
In the planar case, one can also consider a Newtonian potential in the nonlocal term, \emph{i.e.}
$$
\int_E\int_E\log\frac{1}{|x-y|}\,\dd x\dd y\,.
$$
It is clear that the infimum of the corresponding functional on $\R^2$ is $-\infty$
(consider, for instance, a minimizing sequence obtained by sending to infinity the distance between the centers of two disjoint balls).
Moreover, also the notion of local minimality considered in Definition~\ref{def:minloc}
becomes meaningless in this situation, since, given any finite perimeter set $E$,
it is always possible to find sets with total energy arbitrarily close to $-\infty$ in every $L^1$-neighbourhood of $E$.
Nevertheless, by reproducing the arguments of this paper one can show that, given a bounded regular critical set $E$ with positive second variation, and a radius $R>0$ such that $E\subset B_R$,
there exists $\delta>0$ such that $E$ minimizes the energy with respect to competitors $F\subset B_R$ with $\alpha(F,E)<\delta$.
\end{remark}


\section{Second variation and $W^{2,p}$-local minimality} \label{sect:2var}

We start this section by introducing the notions of first and second variation of the functional $\f$
along families of deformations as in the following definition.

\begin{definition} \label{def:flow}
Let $X:\R^N\rightarrow\R^N$ be a $C^2$ vector field.
The \emph{admissible flow} associated with $X$ is the function $\Phi:\R^N\times(-1,1)\rightarrow\R^N$ defined by the equations
$$
\frac{\partial\Phi}{\partial t}=X(\Phi)\,,
\quad
\Phi(x,0)=x\,.
$$
\end{definition}

\begin{definition}
Let $E\subset\R^N$ be a set of class $C^2$, and let $\Phi$ be an admissible flow. We define the \emph{first and second variation of $\f$ at $E$ with respect to the flow $\Phi$} to be
$$
\frac{\dd}{\dd t}\f(E_t)_{|_{t=0}}
\qquad\mbox{and}\qquad
\frac{\dd^2}{\dd t^2}\f(E_t)_{|_{t=0}}
$$
respectively, where we set $E_{t}:=\Phi_t(E)$.
\end{definition}

Given a regular set $E$, we denote by $X_\tau:=X-\langle X,\nu_E\rangle\nu_E$
the tangential part to $\partial E$ of a vector field $X$.
We recall that the tangential gradient $D_\tau$ is defined by $D_\tau\vphi:=(D\vphi)_\tau$,
and that $B_{\partial E}:=D_\tau\nu_E$ is the second fundamental form of $\partial E$.

The following theorem contains the explicit formula for the first and second variation of $\f$.
The computation, which is postponed to the Appendix, is performed by a regularization approach
which is slightly different from the technique used, in the case $\alpha=N-2$,
in \cite{ChoSte} (for a critical set, see also \cite{Mur2}) and in \cite{AceFusMor} (for a general regular set):
here we introduce a family of regularized potentials (depending on a small parameter $\delta\in\R$)
to avoid the problems in the differentiation of the singularity in the nonlocal part,
recovering the result by letting the parameter tend to 0.

\begin{theorem} \label{teo:var}
Let $E\subset\R^N$ be a bounded set of class $C^2$,
and let $\Phi$ be the admissible flow associated with a $C^2$ vector field $X$.
Then the first variation of $\f$ at $E$ with respect to the flow $\Phi$ is
\begin{equation} \label{eq:firvar}
{\frac{\dd\f(E_t)}{\dd t}}_{|_{t=0}} = \int_{\partial E} (H_{\partial E} + 2\gamma v_{E})\langle X, \nu_{E} \rangle \,\dd\hn\,,
\end{equation}
and the second variation of $\f$ at $E$ with respect to the flow $\Phi$ is
\begin{align*}
{\frac{\dd^2 \f(E_t)}{\dd t^2}}_{|_{t=0}}
= &\int_{\partial E} \left( |D_{\tau}\langle X,\nu_{E}\rangle|^2-|B_{\partial E}|^2\langle X,\nu_{E}\rangle^2 \right) \,\dd\hn \\
& + 2\gamma \int_{\partial E}\int_{\partial E}G(x,y) \langle X(x),\nu_{E}(x) \rangle\langle X(y),\nu_{E}(y) \rangle \,\dd\hn(x)\dd\hn(y)  \\
& +2\gamma\int_{\partial E} \partial_{\nu_{E}}v_{E}\,\langle X,\nu_{E}\rangle^2\,\dd\hn
-\int_{\partial E} (2\gamma v_{E}+H_{\partial E})\,\div_{\tau}\bigl(X_{\tau}\langle X,\nu_{E}\rangle\bigr) \,\dd\hn \\
& +\int_{\partial E}(2\gamma v_{E}+H_{\partial E})(\div X)\langle X,\nu_{E}\rangle \,\dd\hn\,,
\end{align*}
where $G(x,y):=\frac{1}{|x-y|^\alpha}$ is the potential in the nonlocal part of the energy.
\end{theorem}

If $E$ is a regular critical set (as in Definition \ref{def:crit}) it holds
$$
\int_{\partial E} (2\gamma v_{E}+H_{\partial E})\div_{\tau}\bigl(X_{\tau}\langle X,\nu_{E}\rangle\bigr) \,\dd\hn = 0\,.
$$
Moreover if the admissible flow $\Phi$ preserves the volume of $E$,
\emph{i.e.} if $|\Phi_t(E)|=|E|$ for all $t\in(-1,1)$,
then (see \cite[equation~(2.30)]{ChoSte})
$$
0=\frac{\dd^2}{\dd t^2}|E_t|_{|_{t=0}} = \int_{\partial E}(\div X)\langle X, \nu_E \rangle \,\dd\hn\,.
$$
Hence we obtain the following expression for the second variation at a regular critical set
with respect to a volume-preserving admissible flow:
\begin{align*}
{\frac{\dd^2 \f(E_t)}{\dd t^2}}_{|_{t=0}}
& =\int_{\partial E} \bigl( |D_{\tau}\langle X,\nu_{E}\rangle|^2-|B_{\partial E}|^2\langle X,\nu_{E}\rangle^2 \bigl) \,\dd\hn
    +2\gamma\int_{\partial E} \partial_{\nu_{E}}v_{E} \langle X,\nu_{E}\rangle^2\,\dd\hn \\
& \quad +2\gamma\int_{\partial E}\int_{\partial E}G(x,y) \langle X(x),\nu_{E}(x) \rangle\langle X(y),\nu_{E}(y) \rangle\,\dd\hn(x)\dd\hn(y)\,.
\end{align*}

Following \cite{AceFusMor}, we introduce the space
$$
\widetilde{H}^1(\partial E):=\left\{\vphi\in H^1(\partial E) \::\: \int_{\partial E}\vphi\,\dd\hn=0 \right\}
$$
endowed with the norm $\|\vphi\|_{\widetilde{H}^1(\partial E)} := \|\nabla\vphi\|_{L^2(\partial E)}$,
and we define on it the following quadratic form associated with the second variation.

\begin{definition} \label{def:fquad}
Let $E\subset\R^N$ be a regular critical set.
We define the quadratic form $\partial^2\f(E):\widetilde{H}^1(\partial E)\rightarrow\R$ by
\begin{equation} \label{eq:fquad}
\begin{split}
\partial^2\f(E)[\vphi]
& =\int_{\partial E} \bigl( |D_{\tau}\vphi|^2-|B_{\partial E}|^2\vphi^2 \bigl) \,\dd\hn
    +2\gamma\int_{\partial E} (\partial_{\nu_{E}}v_{E}) \vphi^2\,\dd\hn \\
& \quad +2\gamma\int_{\partial E}\int_{\partial E}G(x,y) \vphi(x) \vphi(y)\,\dd\hn(x)\dd\hn(y)\,.
\end{split}
\end{equation}
\end{definition}

Notice that if $E$ is a regular critical set and $\Phi$ preserves the volume of $E$, then
\begin{equation} \label{eq:fquadvar2}
\partial^2\f(E)[\langle X, \nu_E \rangle] = {\frac{\dd^2 \f(E_t)}{\dd t^2}}_{|_{t=0}}\,.
\end{equation}
We remark that the last integral in the expression of $\partial^2\f(E)$ is well defined for $\vphi\in\widetilde{H}^1(\partial E)$,
thanks to the following result.

\begin{lemma} \label{lemma:potenziale}
Let $E$ be a bounded set of class $C^1$.
There exists a constant $C>0$, depending only on $E,$ $N$ and $\alpha$, such that
for every $\vphi,\psi\in \widetilde{H}^1(\partial E)$
\begin{equation} \label{eq:potenziale}
\int_{\partial E}\int_{\partial E}G(x,y) \vphi(x)\psi(y) \,\dd\hn(x)\dd\hn(y)
\leq C \|\vphi\|_{L^2}\|\psi\|_{L^2}
\leq C \|\vphi\|_{\widetilde{H}^1}\|\psi\|_{\widetilde{H}^1}\,.
\end{equation}
\end{lemma}

\begin{proof}
The proof lies on \cite[Lemma 7.12]{GT}, which states that if $\Omega\subset\R^n$ is a bounded domain and $\mu\in(0,1]$,
the operator $f \mapsto V_{\mu}f$ defined by
$$
(V_{\mu}f)(x):=\int_{\Omega} |x-y|^{n(\mu-1)}f(y)\,\dd y
$$
maps $L^p(\Omega)$ continuously into $L^{q}(\Omega)$ provided that
$0\leq \delta := p^{-1}-q^{-1}<\mu$, and
$$
\|V_\mu f\|_{L^q(\Omega)} \leq  \Bigl(\frac{1-\delta}{\mu-\delta}\Bigr)^{1-\delta}\omega_n^{1-\mu}|\Omega|^{\mu-\delta} \|f\|_{L^p(\Omega)}\,.
$$
In our case, from the fact that our set has compact boundary,
we can simply reduce to the above case using local charts and partition of unity
(notice that the hypothesis of compact boundary allows us to bound from above in the $L^{\infty}$-norm the area factor).
In particular we have that $\mu=\frac{N-1-\alpha}{N-1}$, and applying this result with $p=q=2$ we easily obtain the estimate in the statement by the Sobolev Embedding Theorem.
\end{proof}

\begin{remark} \label{rm:potenziale}
Using the estimate contained in the previous lemma it is easily seen that $\partial^2\f(E)$ is continuous with respect to the strong convergence in $\widetilde{H}^1(\partial E)$
and  lower semicontinuous with respect to the weak convergence in $\widetilde{H}^1(\partial E)$.
Moreover, it is also clear from the proof that, given $\bar{\alpha}<N-1$,
the constant $C$ in \eqref{eq:potenziale} can be chosen independently of $\alpha\in(0,\bar{\alpha})$.
\end{remark}

Equality \eqref{eq:fquadvar2} suggests that at a regular local minimizer the quadratic form \eqref{eq:fquad} must be nonnegative on the space $\widetilde{H}^1(\partial E)$.
This is the content of the following corollary, whose proof is analogous to \cite[Corollary 3.4]{AceFusMor}.

\begin{corollary} \label{cor:minloc}
Let $E$ be a local minimizer of $\f$ of class $C^2$. Then
$$
\partial^2 \f(E)[\vphi]\geq0\qquad\mbox{for all }\vphi\in\widetilde{H}^1(\partial E)\,.
$$
\end{corollary}

Now we want to look for a sufficient condition for local minimality. First of all we notice that, since our functional is translation invariant, if we compute the second variation of
$\f$ at a regular set $E$ with respect to a flow of the form $\Phi(x,t):=x+t\eta e_i$, where $\eta\in\R$ and $e_i$ is an element of the canonical basis of $\R^N$,
setting $\nu_i:=\langle\nu_E,e_i\rangle$ we obtain that
$$
\partial^2 \f(E)[\eta\nu_i]= {\frac{\dd^2}{\dd t^2}\f(E_t)}_{|_{t=0}} = 0\,.
$$
Following \cite{AceFusMor},
since we aim to prove that the \emph{strict} positivity of the second variation is a sufficient condition for local minimality,
we shall exclude the finite dimensional subspace of $\widetilde{H}^1(\partial E)$ generated by the functions $\nu_i$, which we denote by $T(\partial E )$.
Hence we split
$$
\widetilde{H}^1(\partial E)=T^{\bot}(\partial E)\oplus T(\partial E)\,,
$$
where $T^{\bot}(\partial E)$ is the orthogonal complement to $T(\partial E)$ in the $L^2$-sense, \emph{i.e.},\emph{}
$$
T^{\bot}(\partial E):=\left\{  \vphi\in\widetilde{H}^1(\partial E) \,:\, \int_{\partial E}\vphi\nu_i\,\dd\hn=0\;\;\mbox{for each}\;i=1,\dots,N \right\}\,.
$$
It can be shown (see \cite[Equation~(3.7)]{AceFusMor}) that there exists an orthonormal frame $(\e_1,\dots,\e_N)$ such that
$$
\int_{\partial E}\langle \nu,\e_i \rangle \langle \nu,\e_j \rangle\,\dd\hn=0\qquad\mbox{for all }i\neq j\,,
$$
so that the projection on $T^{\bot}(\partial E)$ of a function $\vphi\in\widetilde{H}^1(\partial E)$ is
$$
{\pi}_{T^{\bot}(\partial E)}(\vphi)
= \vphi-\sum_{i=1}^N \left( \int_{\partial E} \vphi \langle\nu,\e_i\rangle\,\dd\hn \right)
\frac{\langle \nu, \e_i \rangle}{\| \langle\nu, \e_i\rangle\|^2_{L^2(\partial E)}}
$$
(notice that $\langle\nu,\e_i\rangle\not\equiv0$ for every $i$,
since on the contrary the set $E$ would be translation invariant in the direction $\e_i$).

\begin{definition} \label{def:varpos}
We say that $\f$ has \emph{positive second variation} at the regular critical set $E$ if
$$
\partial^2\f(E)[\vphi]>0
\qquad\text{for all }\vphi\in T^\bot(\partial E)\setmeno\{0\}.
$$
\end{definition}

One could expect that the positiveness of the second variation implies also a sort of coercivity;
this is shown in the following lemma.

\begin{lemma}\label{lemma:varpos}
Assume that $\f$ has positive second variation at a regular critical set $E$. Then
$$
m_{0}:=\inf\bigl\{ \partial^2\f(E)[\vphi] \,:\, \vphi\in T^{\bot}(\partial E), \, \|\vphi\|_{\widetilde{H}^1(\partial E)} =1 \bigr\} >0\,,
$$
and
$$
\partial^2\f(E)[\vphi]\geq m_0 \|\vphi\|^2_{\widetilde{H}^1(\partial E)} \qquad\mbox{for all }\vphi\in T^{\bot}(\partial E)\,.
$$
\end{lemma}

\begin{proof}
Let $(\vphi_h)_h$ be a minimizing sequence for $m_0$.
Up to a subsequence we can suppose that $\vphi_h\wto \vphi_0$ weakly in $H^1(\partial E)$, with $\vphi_0\in T^{\bot}(\partial E)$.
By the lower semicontinuity of $\partial^2\f(E)$ with respect to the weak convergence in $H^1(\partial E)$
(see Remark~\ref{rm:potenziale}), we have that if $\varphi_0\neq 0$
$$
m_0=\lim_{h\rightarrow\infty}\partial^2\f(E)[\vphi_h]\geq\partial^2\f(E)[\vphi_0]>0\,,
$$
while if $\vphi_0=0$
$$
m_0=\lim_{h\rightarrow\infty}\partial^2\f(E)[\vphi_h]=\lim_{h\rightarrow\infty}\int_{\partial E}|D_{\tau}\vphi_h|^2\,\dd\hn=1\,.
$$
The second part of the statement follows from the fact that $\partial^2\f(E)$ is a quadratic form.
\end{proof}

We now come to the proof of the main result of the paper, namely that the positivity of the second variation at a critical set $E$
is a sufficient condition for local minimality (Theorem~\ref{teo:minl1}).
In the remaining part of this section we prove that a weaker minimality property holds,
that is minimality with respect to sets whose boundaries are graphs over the boundary of $E$
with sufficiently small $W^{2,p}$-norm (Theorem~\ref{teo:minw}).
In order to do this, we start by recalling a technical result needed in the proof, namely \cite[Theorem~3.7]{AceFusMor},
which provides a construction of an admissible flow connecting a regular set $E\subset\R^N$ with an arbitrary set sufficiently close in the $W^{2,p}$-sense.

\begin{theorem}\label{teo:field}
Let $E\subset\R^N$ be a bounded set of class $C^3$ and let $p>N-1$. For all $\e>0$ there exist a tubular neighbourhood $\mathcal{U}$ of $\partial E$ and two positive constants $\delta,$ $C$ with the following properties: if $\psi\in C^2(\partial E)$ and $\|\psi\|_{W^{2,p}(\partial E)}\leq\delta$ then there exists a field $X\in C^2$ with $\div X=0$ in $\mathcal{U}$ such that
$$
\|X-\psi\nu_E\|_{L^2(\partial E)}\leq\e\|\psi\|_{L^2(\partial E)}\,.
$$
Moreover the associated flow
$$
\Phi(x,0)=0,\;\;\;\frac{\partial\Phi}{\partial t}=X(\Phi)
$$
satisfies $\Phi(\partial E, 1)=\{x+\psi(x)\nu_E(x): x\in\partial E\}$, and for every $t\in[0,1]$
$$
\|\Phi(\cdot,t)-Id\|_{W^{2,p}}\leq C\|\psi\|_{W^{2,p}(\partial E)}\,,
$$
where $Id$ denotes the identity map. If in addition $E_1$ has the same volume as $E$, then for every $t$ we have $|E_t|=|E|$ and
$$
\int_{\partial E_t}\langle X,\nu_{E_t} \rangle\,\dd\hn=0\,.
$$
\end{theorem}

%

We are now in position to prove the following $W^{2,p}$-local minimality theorem, analogous to \cite[Theorem~3.9]{AceFusMor}.
The proof contained in \cite{AceFusMor} can be repeated here with minor changes,
and we will only give a sketch of it for the reader's convenience.

\begin{theorem} \label{teo:minw}
Let $p>\max\{2,N-1\}$ and let $E$ be a regular critical set for $\f$ with positive second variation, according to Definition~\ref{def:varpos}.
Then there exist $\delta,$ $C_0>0$ such that
$$
\f(F)\geq \f(E)+C_0(\alpha(E,F))^2\,,
$$
for each $F\subset\R^N$ such that $|F|=|E|$ and $\partial F=\{x+\psi(x)\nu_E(x): x\in\partial E\}$
with $\|\psi\|_{W^{2,p}(\partial E)}\leq~\delta$.
\end{theorem}

\begin{proof}[Proof (sketch)]
We just describe the strategy of the proof, which is divided into two steps.

\smallskip
\noindent
{\it Step 1.}
There exists $\delta_1>0$ such that if $\partial F=\{x+\psi(x)\nu_E(x)\;:\; x\in\partial E\}$
with $|F|=|E|$ and $\|\psi\|_{W^{2,p}(\partial E)}\leq\delta_1$, then
$$
\inf\left\{ \partial^2\f(F)[\vphi]\;:\;\vphi\in\widetilde{H}^1(\partial F),\,
\|\vphi\|_{\widetilde{H}^1(\partial F)}=1,\,
\Big| \int_{\partial F}\vphi\nu_F\,\dd\hn \Big| \leq\delta_1  \right\}
\geq\frac{m_0}{2},
$$
where $m_0$ is defined in Lemma~\ref{lemma:varpos}.
To prove this we suppose by contradiction that there exist a sequence $(F_n)_n$ of subsets of $\R^N$ such that $\partial F_n=\{x+\psi_n(x)\nu_{E}(x):x\in\partial E\}$, $|F_n|=~|E|$, $\|\psi_n\|_{W^{2,p}(\partial E)}\rightarrow0$, and a sequence of functions $\vphi_n\in\widetilde{H}^1(\partial F_n)$ with $\|\vphi_n\|_{\widetilde{H}^1(\partial F_n)}=1$, $|\int_{\partial F_n}\vphi_n\nu_{F_n}\,\dd\hn|\rightarrow0$, such that
$$
\partial^2\f(F_n)[\vphi_n]<\frac{m_0}{2}.
$$
We consider a sequence of diffeomorphisms $\Phi_n:E\rightarrow F_n$, with $\Phi_n\rightarrow Id$ in $W^{2,p}$, and we set
$$
\tilde{\vphi}_n:=\vphi_n\circ\Phi_n-a_n,\;\;\;\;\;\;a_n:=\med_{\partial E}\vphi_n\circ\Phi_n\,\dd\hn.
$$
Hence $\tilde{\vphi}_n\in\widetilde{H}^1(\partial E)$, $a_n\rightarrow0$, and since $\nu_{F_n}\circ\Phi_n-\nu_E\rightarrow0$ in $C^{0,\beta}$ for some $\beta\in(0,1)$ and a similar convergence holds for the tangential vectors, we have that
$$
\int_{\partial E}\tilde{\vphi}_n\langle \nu_E, \e_i \rangle\,\dd\hn\rightarrow0
$$
for every $i=1,\ldots,N$,
so that $\|\pi_{T^{\bot}(\partial E)}(\tilde{\vphi}_n)\|_{\widetilde{H}^1(\partial E)}\rightarrow1$.
Moreover it can be proved that
$$
\left|\partial^2\f(F_n)[\vphi_n]-\partial^2\f(E)[\tilde{\vphi}_n]\right|\rightarrow0.
$$
Indeed, the convergence of the first integral in the expression of the quadratic form
follows easily from the fact that $B_{\partial F_n}\circ\Phi_n - B_{\partial E}\to0$ in $L^p(\partial E)$,
and from the Sobolev Embedding Theorem (recall that $p>\max\{2,N-1\}$).
For the second integral, it is sufficient to observe that, as a consequence of Proposition~\ref{prop:ve},
the functions $v_{F_h}$ are uniformly bounded in $C^{1,\beta}(\R^N)$ for some $\beta\in(0,1)$
and hence they converge to $v_E$ in $C^{1,\gamma}(B_R)$ for all $\gamma<\beta$ and $R>0$.
Finally, the difference of the last integrals can be written as
\begin{align*}
\int_{\partial F_n}\int_{\partial F_n} &G(x,y)\vphi_n(x)\vphi_n(y)\,\dd\hn\dd\hn
    - \int_{\partial E}\int_{\partial E} G(x,y)\tilde{\vphi}_n(x)\tilde{\vphi}_n(y)\,\dd\hn\dd\hn \\
&= \int_{\partial E}\int_{\partial E} g_n(x,y) G(x,y)\tilde{\vphi}_n(x)\tilde{\vphi}_n(y)\,\dd\hn\dd\hn \\
&    \qquad + a_n \int_{\partial E}\int_{\partial E} G(\Phi_n(x),\Phi_n(y))J_n(x)J_n(y) \bigl( \tilde{\vphi}_n(x) + \tilde{\vphi}_n(y) + a_n \bigr)\,\dd\hn\dd\hn
\end{align*}
where $J_n(z):=J^{N-1}_{\partial E}\Phi_n(z)$ is the $(N-1)$-dimensional jacobian of $\Phi_n$ on $\partial E$, and
$$
g_n(x,y):= \frac{|x-y|^\alpha}{|\Phi_n(x)-\Phi_n(y)|^\alpha}\, J_n(x)J_n(y) - 1\,.
$$
Thus the desired convergence follows from the fact that $g_n\to0$ uniformly, $a_n\to0$,
and from the estimate provided by Lemma~\ref{lemma:potenziale}.

Hence
\begin{align*}
\frac{m_0}{2} \geq \lim_{n\rightarrow\infty}\partial^2\f(F_n)[\vphi_n] &= \lim_{n\rightarrow\infty}\partial^2\f(E)[\tilde{\vphi}_n]
= \lim_{n\rightarrow\infty}\partial^2\f(E)[\pi_{T^{\bot}(\partial E)}(\tilde{\vphi}_n)]\\
&\geq m_0\lim_{n\rightarrow\infty}\|\pi_{T^{\bot}(\partial E)}(\tilde{\vphi}_n)\|_{\widetilde{H}^1(\partial E)}  = m_0,
\end{align*}
which is a contradiction.

\smallskip
\noindent
{\it Step 2.}
If $F$ is as in the statement of the theorem, we can use the vector field $X$ provided by Theorem~\ref{teo:field}
to generate a flow connecting $E$ to $F$ by a family of sets $E_t$, $t\in[0,1]$.
Recalling that $E$ is critical and that $X$ is divergence free, we can write
\begin{align*}
\f(F)&-\f(E)
= \f(E_1) - \f(E_0)
= \int_{0}^{1}(1-t)\frac{\dd^2}{\dd t^2}\f(E_t)\,\dd t  \\
&= \int_{0}^{1}(1-t)\Big( \partial^2\f(E_t)[\langle X,\nu_{E_t} \rangle] - \int_{\partial E_t} (2\gamma v_{E_t} + H_{\partial E_t}) \div_{\tau_t}(X_{\tau_t}\langle X,\nu_{E_t}\rangle) \,\dd\hn \Big)\,\dd t,
\end{align*}
where $\div_{\tau_t}$ stands for the tangential divergence of $\partial E_t$.
It is now possible to bound from below the previous integral in a quantitative fashion:
to do this we use, in particular, the result proved in Step 1 for the first term,
and we proceed as in Step $2$ of \cite[Theorem~3.9]{AceFusMor} for the second one.
In this way we obtain the desired estimate.
\end{proof}


\section{$L^1$-local minimality} \label{sect:l1min}

In this section we complete the proof of the main result of the paper (Theorem~\ref{teo:minl1}),
started in the previous section.
The main argument of the proof relies on a regularity property of sequences of \emph{quasi-minimizers} of the area functional,
which has been observed by White in \cite{Whi} and was implicitly contained in \cite{Alm} (see also \cite{SchSim}, \cite{Tam}).

\begin{definition} \label{def:quamin}
A set $E\subset\R^N$ is said to be an \emph{$(\omega,r_0)$-minimizer} for the area functional, with $\omega>0$ and $r_0>0$,
if for every ball $B_r(x)$ with $r\leq r_0$
and for every finite perimeter set $F\subset\R^N$ such that $E\triangle F\subset\subset B_r(x)$ we have
$$
\p(E) \leq \p(F) + \omega |E\triangle F|.
$$
\end{definition}

\begin{theorem} \label{teo:quamin}
Let $E_n\subset\R^N$ be a sequence of $(\omega,r_0)$-minimizers of the area functional such that
$$
\sup_n\p(E_n)<+\infty
\quad\text{and}\quad
\chi_{E_n}\to\chi_E\text{ in }L^1(\R^N)
$$
for some bounded set $E$ of class $C^2$. Then for $n$ large enough $E_n$ is of class $C^{1,\frac12}$ and
$$
\partial E_n = \{x+\psi_n(x)\nu_E(x) \,:\, x\in\partial E \},
$$
with $\psi_n\to0$ in $C^{1,\beta}(\partial E)$ for all $\beta\in(0,\frac12)$.
\end{theorem}

Another useful result is the following consequence of the classical elliptic regularity theory
(see \cite[Lemma~7.2]{AceFusMor} for a proof).

\begin{lemma} \label{lemma:curvature}
Let $E$ be a bounded set of class $C^2$ and let $E_n$ be a sequence of sets of class $C^{1,\beta}$ for some $\beta\in(0,1)$ such that
$\partial E_n = \{x+\psi_n(x)\nu_E(x) \,:\, x\in\partial E \}$, with $\psi_n\to0$ in $C^{1,\beta}(\partial E)$.
Assume also that $H_{\partial E_n}\in L^p(\partial E_n)$ for some $p\geq1$.
If
$$
H_{\partial E_n}(\cdot + \psi_n(\cdot)\nu_E(\cdot))\to H_{\partial E} \qquad\text{in }L^p(\partial E),
$$
then $\psi_n\to0$ in $W^{2,p}(\partial E)$.
\end{lemma}

We recall also the following simple lemma from \cite[Lemma~4.1]{AceFusMor}.

\begin{lemma} \label{lemma:stimaper}
Let $E\subset\R^N$ be a bounded set of class $C^2$.
Then there exists a constant $C_E>0$, depending only on $E$, such that for every finite perimeter set $F\subset\R^N$
$$
\p(E)\leq \p(F) + C_E|E\triangle F|.
$$
\end{lemma}

An intermediate step in the proof of Theorem~\ref{teo:minl1}
consists in showing that the $W^{2,p}$-local minimality proved in Theorem~\ref{teo:minw}
implies local minimality with respect to competing sets which are sufficiently close in the Hausdorff distance.
We omit the proof of this result, since it can be easily adapted from \cite[Theorem~4.3]{AceFusMor}
(notice, indeed, that the difficulties coming from the fact of working in the whole space $\R^N$ are not present,
due to the constraint $F\subset\mathcal{I}_{\delta_0}(E)$).

\begin{theorem} \label{teo:mininf}
Let $E\subset\R^N$ be a bounded regular set, and assume that there exists $\delta>0$ such that
\begin{equation} \label{eq:mininf}
\f(E)\leq\f(F)
\end{equation}
for every set $F\subset\R^N$ with $|F|=|E|$ and $\partial F = \{x+\psi(x)\nu_E(x):x\in\partial E\}$,
for some function $\psi$ with $\|\psi\|_{W^{2,p}(\partial E)}\leq\delta$.

Then there exists $\delta_0>0$ such that \eqref{eq:mininf} holds for every finite perimeter set $F$ with $|F|=|E|$
and such that $\I_{-\delta_0}(E)\subset F \subset\I_{\delta_0}(E)$,
where for $\delta\in\R$ we set ($d$ denoting the signed distance to $E$)
$$
\I_\delta(E):=\{x:d(x)<\delta\}\,.
$$
\end{theorem}

We are finally ready to complete the proof of the main result of the paper.
The strategy follows closely \cite[Theorem~1.1]{AceFusMor},
with the necessary technical modifications due to the fact that
here we have to deal with a more general exponent $\alpha$
and with the lack of compactness of the ambient space.

\begin{proof}[Proof of Theorem~\ref{teo:minl1}]
We assume by contradiction that there exists a sequence of sets $E_h\subset\R^N$, with $|E_h|=|E|$ and $\alpha(E_h,E)>0$,
such that $\e_h:=\alpha(E_h,E)\to0$ and
\begin{equation} \label{contraminL1}
\f(E_h) < \f(E) + \frac{C_0}{4} \bigl( \alpha(E_h,E) \bigr)^2,
\end{equation}
where $C_0$ is the constant provided by Theorem~\ref{teo:minw}.
By approximation we can assume without loss of generality that each set of the sequence is bounded,
that is, there exist $R_h>0$ (which we can also take satisfying $R_h\to+\infty$)
such that $E_h\subset B_{R_h}$.

We now define $F_h\subset\R^N$ as a solution to the penalization problem
\begin{equation} \label{penalizzatoL1}
\min \left\{
\mathcal{J}_h(F) := \f(F) + \Lambda_1\sqrt{\bigl( \alpha(F,E)-\e_h \bigr)^2+\e_h} + \Lambda_2\big||F|-|E|\big|
\,:\, F\subset B_{R_h}
\right\},
\end{equation}
where $\Lambda_1$ and $\Lambda_2$ are positive constant, to be chosen
(notice that the constraint $F\subset B_{R_h}$ guarantees the existence of a solution).
We first fix
\begin{equation} \label{eq:lambda1}
\Lambda_1> C_E + c_0\gamma\,.
\end{equation}
Here $C_E$ is as in Lemma~\ref{lemma:stimaper},
while $c_0$ is the constant provided by Proposition~\ref{prop:NLlipschitz}
corresponding to the fixed values of $N$ and $\alpha$ and to $m:=|E|+1$.
We remark that with this choice $\Lambda_1$ depends only on the set $E$.
We will consider also the sets $\ftil$ obtained by translating $F_h$ in such a way that $\alpha(F_h,E)=|\ftil\triangle E|$
(clearly $\mathcal{J}_h(\ftil)=\mathcal{J}_h(F_h)$).

\smallskip
\noindent
{\it Step 1.}
We claim that, if $\Lambda_2$ is sufficiently large (depending on $\Lambda_1$, but not on $h$),
then $|F_h|=|E|$ for every $h$ large enough.
This can be deduced by adapting an argument from \cite[Section~2]{EspFus} (see also \cite[Proposition~2.7]{AceFusMor}).
Indeed, assume by contradiction that there exist $\Lambda_h\to\infty$
and $F_h$ solution to the minimum problem \eqref{penalizzatoL1} with $\Lambda_2$ replaced by $\Lambda_h$
such that $|F_h|<|E|$ (a similar argument can be performed in the case $|F_h|>|E|$).
Up to subsequences, we have that $F_h\to F_0$ in $L^1_{\rm loc}$ and $|F_h|\to|E|$.

As each set $F_h$ minimizes the functional
$$
\f(F) + \Lambda_1\sqrt{\bigl( \alpha(F,E)-\e_h \bigr)^2+\e_h}
$$
in $B_{R_h}$ under the constraint $|F|=|F_h|$,
it is easily seen that $F_h$ is a quasi-minimizer of the perimeter with volume constraint,
so that by the regularity result contained in \cite[Theorem~1.4.4]{Rig} we have that
the $(N-1)$-dimensional density of $\partial^*F_h$ is uniformly bounded from below by a constant
independent of $h$.
This observation implies that we can assume without loss of generality that the limit set $F_0$ is not empty
and that there exists a point $x_0\in\partial^*F_0$, so that, by repeating an argument contained in \cite{EspFus},
we obtain that given $\e>0$ we can find $r>0$ and $\bar{x}\in\R^N$ such that
$$
|F_h\cap B_{r/2}(\bar{x})| < \e r^N,
\quad
|F_h\cap B_r(\bar{x})| > \frac{\omega_Nr^N}{2^{N+2}}
$$
for every $h$ sufficiently large (and we assume $\bar{x}=0$ for simplicity).

Now we modify $F_h$ in $B_r$ by setting $G_h:=\Phi_h(F_h)$, where $\Phi_h$ is the bilipschitz map
$$
\Phi_h(x):=
\left\{
  \begin{array}{ll}
    \bigl( 1-\sigma_h(2^N-1) \bigr)x                & \hbox{if }|x|\leq\frac{r}{2}, \\
    x+\sigma_h\bigl( 1-\frac{r^N}{|x|^N} \bigr)x    & \hbox{if } \frac{r}{2} < x < r,\\
    x                                               & \hbox{if } |x|\geq r,
  \end{array}
\right.
$$
and $\sigma_h\in(0,\frac{1}{2^N})$.
It can be shown (see \cite[Section~2]{EspFus}, \cite[Proposition~2.7]{AceFusMor} for details)
that $\e$ and $\sigma_h$ can be chosen in such a way that $|G_h|=|E|$, and moreover
there exists a dimensional constant $C>0$ such that
\begin{align*}
J_{\Lambda_h}(F_h)-J_{\Lambda_h}(G_h) \geq
\sigma_h\Bigl( C\Lambda_hr^N - (2^NN+C\gamma+C\Lambda_1)\p(F_h;B_r)\Bigr)
\end{align*}
(where $J_{\Lambda_h}$ denotes the functional in \eqref{penalizzatoL1} with $\Lambda_2$ replaced by $\Lambda_h$).
This contradicts the minimality of $F_h$ for $h$ sufficiently large.

\smallskip
\noindent
{\it Step 2.}
We now show that
\begin{equation} \label{claim1minL1}
\lim_{h\to+\infty}\alpha(F_h,E)=0.
\end{equation}
Indeed, by Lemma~\ref{lemma:stimaper} we have that
$$
\p(E) \leq \p(\ftil) + C_E|\ftil\triangle E|,
$$
while by Proposition~\ref{prop:NLlipschitz}
$$
|\nl(E)-\nl(\ftil)| \leq c_0 |\ftil\triangle E|.
$$
Combining the two estimates above, using the minimality of $F_h$ and recalling that $|F_h|=|E|$ we deduce
\begin{align*}
\p(\ftil) + \gamma\nl(\ftil) &+ \Lambda_1\sqrt{\bigl( |\ftil\triangle E|-\e_h \bigr)^2+\e_h}
    = \mathcal{J}_h(F_h)
    \leq \mathcal{J}_h(E) \\
&   = \p(E) + \gamma\nl(E) + \Lambda_1\sqrt{\e_h^2+\e_h} \\
&   \leq \p(\ftil) + \gamma\nl(\ftil) + (C_E+c_0\gamma)|\ftil\triangle E| + \Lambda_1\sqrt{\e_h^2+\e_h},
\end{align*}
which yields
$$
\Lambda_1\sqrt{\bigl( |\ftil\triangle E|-\e_h \bigr)^2+\e_h}
\leq
(C_E+c_0\gamma)|\ftil\triangle E| + \Lambda_1\sqrt{\e_h^2+\e_h}.
$$
Passing to the limit as $h\to+\infty$, we conclude that
$$
\Lambda_1\limsup_{h\to+\infty}|\ftil\triangle E| \leq (C_E+c_0\gamma)\limsup_{h\to+\infty}|\ftil\triangle E|,
$$
which implies $|\ftil\triangle E|\to0$ by the choice of $\Lambda_1$ in \eqref{eq:lambda1}.
Hence \eqref{claim1minL1} is proved,
and this shows in particular that $\chi_{\ftil}\to\chi_E$ in $L^1(\R^N)$.

\smallskip
\noindent
{\it Step 3.}
Each set $F_h$ is an $(\omega,r_0)$-minimizer of the area functional (see Definition~\ref{def:quamin}),
for suitable $\omega>0$ and $r_0>0$ independent of $h$.
Indeed, choose $r_0$ such that $\omega_N{r_0}^N\leq 1$,
and consider any ball $B_r(x)$ with $r\leq r_0$
and any finite perimeter set $F$ such that $F\triangle F_h\subset\subset B_r(x)$.
We have
$$
|\nl(F)-\nl(F_h)|\leq c_0|F\triangle F_h|
$$
by Proposition~\ref{prop:NLlipschitz}, where $c_0$ is the same constant as before
since we can bound the volume of $F$ by
$|F|\leq |F_h|+\omega_N{r_0}^N\leq |E| + 1$.
Moreover
\begin{align*}
\p(F)-\p(F&\cap B_{R_h})
    = \int_{\partial^*F\setminus B_{R_h}}1\,\dd\hn(x) - \int_{\partial^*(F\cap B_{R_h})\cap\partial B_{R_h}}1\,\dd\hn(x)\\
&   \geq \int_{\partial^*F\setminus B_{R_h}}\frac{x}{|x|}\cdot\nu_F\,\dd\hn(x)
    - \int_{\partial^*(F\cap B_{R_h})\cap\partial B_{R_h}}\frac{x}{|x|}\cdot\nu_{F\cap B_{R_h}}\,\dd\hn(x)\\
&   = \int_{\partial^*(F\setminus B_{R_h})}\frac{x}{|x|}\cdot\nu_{F\setminus B_{R_h}}\,\dd\hn(x)
    = \int_{F\setminus B_{R_h}} \div\,\frac{x}{|x|}\,\dd x \geq0.
\end{align*}
Hence, as $F_h$ is a minimizer of $\mathcal{J}_h$ among sets contained in $B_{R_h}$, we deduce
\begin{align*}
\p(F_h)
&   \leq \p(F\cap B_{R_h}) + \gamma\bigl(\nl(F\cap B_{R_h})-\nl(F_h)\bigr) + \Lambda_2 \big||F\cap B_{R_h}|-|E|\big| \\
&\hspace{.5cm}
    + \Lambda_1 \sqrt{\bigl( \alpha(F\cap B_{R_h},E)-\e_h \bigr)^2+\e_h}
    - \Lambda_1 \sqrt{\bigl( \alpha(F_h,E)-\e_h \bigr)^2+\e_h}\\
&   \leq \p(F) + \bigl(c_0\gamma + \Lambda_1 + \Lambda_2\bigr)|(F\cap B_{R_h})\triangle F_h|\\
&   \leq \p(F) + \bigl(c_0\gamma + \Lambda_1 + \Lambda_2\bigr)|F\triangle F_h|
\end{align*}
for $h$ large enough.
This shows that $F_h$ is an $(\omega,r_0)$-minimizer of the area functional
with $\omega=c_0\gamma+\Lambda_1+\Lambda_2$
(and the same holds obviously also for $\ftil$).

Hence, by Theorem~\ref{teo:quamin} and recalling that $\chi_{\ftil}\to\chi_E$ in $L^1$,
we deduce that for $h$ sufficiently large $\ftil$ is a set of class $C^{1,\frac12}$ and
$$
\partial\ftil = \{ x+\psi_h(x)\nu_E(x) \,:\, x\in\partial E \}
$$
for some $\psi_h$ such that $\psi_h\to0$ in $C^{1,\beta}(\partial E)$ for every $\beta\in(0,\frac12)$.
We remark also that the sets $\ftil$ are uniformly bounded, and for $h$ large enough $\ftil\subset\subset B_{R_h}$:
in particular, $\ftil$ solves the minimum problem \eqref{penalizzatoL1}.

\smallskip
\noindent
{\it Step 4.}
We now claim that
\begin{equation} \label{claim2minL1}
\lim_{h\to+\infty}\frac{\alpha(F_h,E)}{\e_h}=1.
\end{equation}
Indeed, assuming by contradiction that $|\alpha(F_h,E)-\e_h|\geq\sigma\e_h$ for some $\sigma>0$ and for infinitely many $h$,
we would obtain
\begin{align*}
\f(F_h) + \Lambda_1\sqrt{\sigma^2\e_h^2+\e_h}
& \leq \f(F_h) + \Lambda_1\sqrt{\bigl(\alpha(F_h,E)-\e_h\bigr)^2+\e_h} \\
& \leq \f(E_h) + \Lambda_1\sqrt{\e_h}
< \f(E) + \frac{C_0}{4}\,\e_h^2 + \Lambda_1\sqrt{\e_h} \\
& \leq \f(\ftil) + \frac{C_0}{4}\,\e_h^2 + \Lambda_1\sqrt{\e_h}
\end{align*}
where the second inequality follows from the minimality of $F_h$,
the third one from \eqref{contraminL1}
and the last one from Theorem~\ref{teo:mininf}.
This shows that
$$
\Lambda_1\sqrt{\sigma^2\e_h^2+\e_h} \leq \frac{C_0}{4}\,\e_h^2 + \Lambda_1\sqrt{\e_h}\,,
$$
which is a contradiction for $h$ large enough.

\smallskip
\noindent
{\it Step 5.}
We now show the existence of constants $\lambda_h\in\R$ such that
\begin{equation}\label{claim3minL1}
\| H_{\partial\ftil} + 2\gamma v_{\ftil} - \lambda_h \|_{L^\infty(\partial\ftil)}\leq 4\Lambda_1\sqrt{\e_h}\to0.
\end{equation}
We first observe that the function $f_h(t):=\sqrt{(t-\e_h)^2+\e_h}$ satisfies
\begin{equation} \label{claim4minL1}
|f_h(t_1)-f_h(t_2)| \leq 2\sqrt{\e_h}\,|t_1-t_2|
\qquad\text{if}\quad |t_i-\e_h|\leq\e_h.
\end{equation}
Hence for every set $F\subset\R^N$ with $|F|=|E|$, $F\subset B_{R_h}$ and $|\alpha(F,E)-\e_h|\leq\e_h$
we have
\begin{align}\label{claim5minL1}
\f(\ftil)
& \leq \f(F) + \Lambda_1
    \Bigl( \sqrt{\bigl(\alpha(F,E)-\e_h\bigr)^2+\e_h} - \sqrt{\bigl(\alpha(\ftil,E)-\e_h\bigr)^2+\e_h} \Bigr) \nonumber\\
& \leq \f(F) + 2\Lambda_1\sqrt{\e_h}\,|\alpha(F,E)-\alpha(\ftil,E)| \\
& \leq \f(F) + 2\Lambda_1\sqrt{\e_h}\,|F\triangle\ftil| \nonumber
\end{align}
where we used the minimality of $\ftil$ in the first inequality,
and \eqref{claim4minL1} combined with the fact that $|\alpha(\ftil,E)-\e_h|\leq\e_h$ for $h$ large (which, in turn, follows by \eqref{claim2minL1}) in the second one.

Consider now any variation $\Phi_t$, as in Definition~\ref{def:flow},
preserving the volume of the set $\ftil$,
associated with a vector field $X$.
For $|t|$ sufficiently small we can plug the set $\Phi_t(\ftil)$ in the inequality \eqref{claim5minL1}:
$$
\f(\ftil) \leq \f(\Phi_t(\ftil)) + 2\Lambda_1\sqrt{\e_h}\,|\Phi_t(\ftil)\triangle\ftil|,
$$
which gives
$$
\f(\Phi_t(\ftil)) - \f(\ftil) + 2\Lambda_1\sqrt{\e_h}\,|t|\int_{\partial\ftil}|X\cdot\nu_{\ftil}|\,\dd\hn + o(t)\geq0
$$
for $|t|$ sufficiently small.
Hence, dividing by $t$ and letting $t\to0^+$ and $t\to0^-$, we get
$$
\bigg|\int_{\partial\ftil} \bigl( H_{\partial\ftil} + 2\gamma v_{\ftil} \bigr) X\cdot\nu_{\ftil}\,\dd\hn\bigg|
\leq
2\Lambda_1\sqrt{\e_h} \int_{\partial\ftil}|X\cdot\nu_{\ftil}|\,\dd\hn,
$$
and by density
$$
\bigg|\int_{\partial\ftil} \bigl( H_{\partial\ftil} + 2\gamma v_{\ftil} \bigr) \vphi\,\dd\hn\bigg|
\leq
2\Lambda_1\sqrt{\e_h} \int_{\partial\ftil}|\vphi|\,\dd\hn
$$
for every $\vphi\in C^\infty(\partial\ftil)$ with $\int_{\partial\ftil}\vphi\,\dd\hn=0$.
In turn, this implies \eqref{claim3minL1} by a simple functional analysis argument.

\smallskip
\noindent
{\it Step 6.}
We are now close to the end of the proof.
Recall that on $\partial E$
\begin{equation}\label{ELL11}
H_{\partial E} = \lambda - 2\gamma v_E
\end{equation}
for some constant $\lambda$, while by \eqref{claim3minL1}
\begin{equation}\label{ELL12}
H_{\partial \ftil} = \lambda_h - 2\gamma v_{\ftil} + \rho_h,
\qquad\text{with }\rho_h\to0\text{ uniformly.}
\end{equation}
Observe now that, since the functions $v_{\ftil}$ are equibounded in $C^{1,\beta}(\R^N)$ for some $\beta\in(0,1)$
(see Proposition~\ref{prop:ve}) and they converge pointwise to $v_E$ since $\chi_{\ftil}\to\chi_E$ in $L^1$, we have that
\begin{equation} \label{ELL13}
v_{\ftil}\to v_E
\qquad\text{in }C^1(\overline{B}_R)\text{ for every }R>0.
\end{equation}

We consider a cylinder $C=B'\times]-L,L[$, where $B'\subset\R^{N-1}$ is a ball centered at the origin,
such that in a suitable coordinate system we have
\begin{align*}
\ftil\cap C   &= \{ (x',x_N)\in C : x'\in B',\, x_N<g_h(x') \},\\
E\cap C     &= \{ (x',x_N)\in C : x'\in B',\, x_N<g(x') \}
\end{align*}
for some functions $g_h\to g$ in $C^{1,\beta}(\overline{B'})$ for every $\beta\in(0,\frac12)$.
By integrating \eqref{ELL12} on $B'$ we obtain
\begin{align*}
\lambda_h&\LL^{N-1}(B') - 2\gamma\int_{B'} v_{\ftil}(x',g_h(x'))\,\dd\LL^{N-1}(x') + \int_{B'}\rho_h(x',g_h(x'))\,\dd\LL^{N-1}(x')\\
&   = - \int_{B'} \div \biggl( \frac{\nabla g_h}{\sqrt{1+|\nabla g_h|^2}} \biggr) \,\dd\LL^{N-1}(x')
    = -\int_{\partial B'} \frac{\nabla g_h}{\sqrt{1+|\nabla g_h|^2}} \cdot \frac{x'}{|x'|} \,\dd\mathcal{H}^{N-2}\,,
\end{align*}
and the last integral in the previous expression converges as $h\to0$ to
\begin{align*}
-\int_{\partial B'} \frac{\nabla g}{\sqrt{1+|\nabla g|^2}} &\cdot \frac{x'}{|x'|} \,\dd\mathcal{H}^{N-2}
    = - \int_{B'} \div \biggl( \frac{\nabla g}{\sqrt{1+|\nabla g|^2}} \biggr) \,\dd\LL^{N-1}(x') \\
&   = \lambda\LL^{N-1}(B') - 2\gamma\int_{B'} v_{E}(x',g(x'))\,\dd\LL^{N-1}(x')\,,
\end{align*}
where the last equality follows by \eqref{ELL11}.
This shows, recalling \eqref{ELL13} and that $\rho_h$ tends to 0 uniformly, that $\lambda_h\to\lambda$,
which in turn implies, by \eqref{ELL11}, \eqref{ELL12} and \eqref{ELL13},
$$
H_{\partial\ftil}(\cdot+\psi_h(\cdot)\nu_E(\cdot))\to H_{\partial E}
\qquad\text{in }L^\infty(\partial E).
$$
By Lemma~\ref{lemma:curvature} we conclude that
$\psi_h\in W^{2,p}(\partial E)$ for every $p\geq1$ and $\psi_h\to0$ in $W^{2,p}(\partial E)$.

Finally, by minimality of $\ftil$ we have
\begin{align*}
\f(\ftil)
& \leq \f(\ftil) + \Lambda_1\sqrt{\bigl( \alpha(\ftil,E)-\e_h \bigr)^2+\e_h} - \Lambda_1\sqrt{\e_h} \\
& \leq \f(E_h)
< \f(E) + \frac{C_0}{4}\,\e_h^2
\leq \f(E) + \frac{C_0}{2}\,\bigl(\alpha(\ftil,E)\bigr)^2
\end{align*}
where we used \eqref{contraminL1} in the third inequality
and \eqref{claim2minL1} in the last one.
This is the desired contradiction with the conclusion of Theorem~\ref{teo:minw}.
\end{proof}

\begin{remark}
It is important to remark that in the arguments of this section
we have not made use of the assumption of strict positivity of the second variation:
the quantitative $L^1$-local minimality follows in fact just from the $W^{2,p}$-local minimality.
\end{remark}


\section{Local minimality of the ball} \label{sect:ball}

In this section we will obtain Theorem~\ref{teo:minlocpalla} as a consequence of Theorem~\ref{teo:minl1},
by computing the second variation of the ball and studying the sign of the associated quadratic form.

\subsection{Recalls on spherical harmonics}

We first recall some basic facts about spherical harmonics, which are needed in our calculation.
We refer to \cite{Gro} for an account on this topic.

\begin{definition}
A \emph{spherical harmonic of dimension} $N$ is the restriction to $S^{N-1}$ of a \emph{harmonic polynomial} in $N$ variables,
\emph{i.e.} a homogeneous polynomial $p$ with $\Delta p=0$.
\end{definition}

We will denote by $\mathcal{H}^{N}_{d}$ the set of all spherical harmonics of dimension $N$
that are obtained as restrictions to $S^{N-1}$ of homogeneous polynomials of degree $d$.
In particular $\mathcal{H}^{N}_{0}$ is the space of constant functions, and $\mathcal{H}^{N}_{1}$ is generated by the coordinate functions. The basic properties of spherical harmonics that we need are listed in the following theorem.

\begin{theorem} \label{teo:sphar}
It holds:

\begin{enumerate}
\item for each $d\in \N$, $\mathcal{H}^{N}_{d}$ is a finite dimensional vector space.

\item If $F\in \mathcal{H}^{N}_{d}$, $G\in \mathcal{H}^{N}_{e}$ and $d\neq e$, then $F$ and $G$ are orthogonal (in the $L^2$-sense).

\item If $F\in \mathcal{H}^{N}_{d}$ and $d\neq 0$, then
			$$
				\int_{S^{N-1}} F\,\dd \hn=0.
			$$
			
\item If $( H_{d}^{1},\dots, H_{d}^{\dim(\mathcal{H}^{N}_{d})})$ is an orthonormal basis of $\mathcal{H}^{N}_{d}$ for every $d\geq0$, then this sequence is complete, \emph{i.e.} every $F\in L^2(S^{N-1})$ can be written in the form
			\begin{equation} \label{eq:sph}
				F=\sum_{d=0}^{\infty} \sum_{i=1}^{\dim(\mathcal{H}^{N}_{d})} c_{d}^{i} H_{d}^{i}\,,
			\end{equation}
    where $c_{d}^{i}:=\langle F, H_{d}^{i} \rangle_{L^2}$.

\item If $H_{d}^{i}$ are as in (4) and $F, G\in L^2(S^{N-1})$ are such that
			$$
			F=\sum_{d=0}^{\infty} \sum_{i=1}^{\dim(\mathcal{H}^{N}_{d})} c_{d}^{i} H_{d}^{i} \,,
            \qquad
            G=\sum_{d=0}^{\infty} \sum_{i=1}^{\dim(\mathcal{H}^{N}_{d})} e_{d}^{i} H_{d}^{i}\,,
			$$
            then
			$$
			\langle F, G \rangle_{L^{2}} = \sum_{d=0}^{\infty} \sum_{i=1}^{\dim(\mathcal{H}^{N}_{d})} c_{d}^{i}e_{d}^{i}.
			$$

\item Spherical harmonics are eigenfunctions of the \emph{Laplace-Beltrami operator} $\Delta_{S^{N-1}}$. More precisely, if $H\in\mathcal{H}^N_d$ then
            $$
            -\Delta_{S^{N-1}}H = d(d+N-2)H.
            $$

\item If $F$ is a $C^2$ function on $S^{N-1}$ represented as in \eqref{eq:sph}, then
			$$
			\int_{S^{N-1}} |D_{\tau}F|^2\,\dd \hn(x) \,=\, \sum_{d=0}^{\infty}\sum_{i=1}^{\dim(\mathcal{H}^{N}_{d})}d(d+N-2)(c_{d}^{i})^2.
			$$
\end{enumerate}
\end{theorem}

We recall also the following important result in the theory of spherical harmonics.

\begin{theorem}[Funk-Hecke Formula] \label{teo:fuhe}
Let $f: (-1,1)\rightarrow \R$ such that
$$
\int_{-1}^{1} |f(t)|(1-t^2)^\frac{N-3}{2} \,\dd t < \infty\,.
$$
Then if $H\in \mathcal{H}^{N}_{d}$ and $x_{0}\in S^{N-1}$ it holds
$$
\int_{S^{N-1}} f(\langle x_{0}, x \rangle) H(x) \,\dd\hn(x) \,=\, \mu_{d}H(x_{0}) \,,
$$
where the coefficient $\mu_{d}$ is given by
$$
\mu_{d} = (N-1)\omega_{N-1}\int_{-1}^{1} P_{N,d}(t)f(t)(1-t^2)^{\frac{N-3}{2}}\,\dd t.
$$
Here $P_{N,d}$ is the \emph{Legendre polynomial} of dimension $N$ and degree $d$ given by
$$
P_{N,d}(t) = (-1)^{d}\frac{\Gamma(\frac{N-1}{2})}{2^{d}\Gamma(d+\frac{N-1}{2})}(1-t^2)^{-\frac{N-3}{2}}\Big( \frac{\dd}{\dd t} \Big)^d(1-t^2)^{d+\frac{N-3}{2}}\,,
$$
where $\Gamma(x):=\int_{0}^{\infty}t^{x-1}e^{-t}\dd t$ is the Gamma function.
\end{theorem}

\subsection{Second variation of the ball}

The quadratic form \eqref{eq:fquad} associated with the second variation of $\f$ at the ball $B_{R}$, computed at a function $\tilde{\vphi}\in \widetilde{H}^{1}(\partial B_{R})$ is
\begin{align*}
\partial^{2}\f(B_{R})[\tilde{\vphi}] & =
    \int_{\partial B_{R}}\Bigl( |D_{\tau}\tilde{\vphi}(x)|^2- \frac{N-1}{R^2}\,\tilde{\vphi}^2(x) \Bigr)\,\dd \hn(x) \\
													   &\hspace{0.5cm} +\; 2\gamma\,\int_{\partial B_{R}}\int_{\partial B_{R}} \frac{1}{|x-y|^\alpha}\tilde{\vphi}(x)\tilde{\vphi}(y)\,\dd \hn(x)\,\dd \hn(y) \\
														&\hspace{0.5cm} +\; 2\gamma\,\int_{\partial B_{R}}\Big( \int_{B_{R}} -\alpha\frac{\langle x-y, \frac{x}{|x|}\rangle}{|x-y|^{\alpha+2}} \,\dd y \Big)\tilde{\vphi}^2(x)\,\dd
																	\hn(x).
\end{align*}
Since we want to obtain a sign condition of $\partial^{2}\f(B_{R})[\tilde{\vphi}]$ in terms of the radius $R$, we first make a change of variable:
\begin{align}\label{eq:secvarball}
\partial^{2}\f(B_{R})[\tilde{\vphi}] & = R^{N-3}\int_{\partial B_{1}}(|D_{\tau}\vphi(x)|^2-(N-1)\vphi^2(x))\,\dd \hn(x) \nonumber \\
														&\hspace{0.5cm} +\; 2\gamma R^{2N-2-\alpha}\,\int_{\partial B_{1}}\int_{\partial B_{1}} \frac{1}{|x-y|^\alpha}\vphi(x)\vphi(x)\,\dd \hn(x)\,\dd \hn(y) \\
														&\hspace{0.5cm} +\; 2\gamma R^{2N-2-\alpha}\,\int_{\partial B_{1}}\Big( \int_{B_{1}} -\alpha\frac{\langle x-y, x\rangle}{|x-y|^{\alpha+2}} \,\dd y \Big)\vphi^2(x)	
																		\,\dd \hn(x), \nonumber
\end{align}
where the function $\vphi\in\widetilde{H}^1(S^{N-1})$ is defined as $\vphi(x):=\tilde{\vphi}(Rx)$.
Since we are only interested in the sign of the second variation, which is continuous with respect to the strong convergence in $\widetilde{H}^{1}(S^{N-1})$, we can assume $\vphi \in C^{2}(S^{N-1})\cap T^\bot(S^{N-1})$.

The idea to compute the second variation at the ball is to expand $\vphi$ with respect to an orthonormal basis of spherical harmonics, as in \eqref{eq:sph}. First of all we notice that if $\vphi \in T^\bot(S^{N-1})$, then its harmonic expansion does not contain spherical harmonics of order $0$ and $1$. Indeed, harmonics of order $0$ are constant functions, that are not allowed by the null average condition. Moreover $\mathcal{H}^{N}_{1} = T(S^{N-1})$, because $\nu_{S^{N-1}}(x)=x$, and the functions $x_{i}$ form an orthonormal basis of $\mathcal{H}_{1}^{N}$.
Hence we can write the harmonic expansion of $\vphi\in C^{2}(S^{N-1})\cap T^\bot(S^{N-1})$ as follows:
$$
\vphi = \sum_{d=2}^{\infty} \sum_{i=1}^{\dim(\mathcal{H}^{N}_{d})} c_{d}^{i} H_{d}^{i}\,,
$$
where $( H_{d}^{1},\dots, H_{d}^{\dim(\mathcal{H}^{N}_{d})})$ is an orthonormal basis of $\mathcal{H}^{N}_{d}$ for each $d\in \N$.
We can now compute each term appearing in \eqref{eq:secvarball} as follows:
the first term, by property (7) of Theorem~\ref{teo:sphar}, is
$$
\int_{\partial B_{1}}(|D_{\tau}\vphi|^2-(N-1)\vphi^2)\,\dd \hn = \sum_{d=2}^{\infty} \sum_{i=1}^{\dim(\mathcal{H}^{N}_{d})} \bigl(d(d+N-2) - (N-1)\bigr)(c_{d}^{i})^2.
$$
For the second term we want to use the Funk-Hecke Formula to compute the inner integral; so we define the function
$$
f(t):=\Big( 2(1-t) \Big)^{-\frac{\alpha}{2}}
$$
and we notice that
$$
|x-y|^{-\alpha}= f(\langle x,y \rangle) \qquad\mbox{for } x,y\in S^{N-1}\,,
$$
and that, for $\alpha\in(0,N-1)$, $f$ satisfies the integrability assumptions of Theorem \ref{teo:fuhe}. Hence for each $y\in S^{N-1}$
$$
\int_{\partial B_{1}} \frac{1}{|x-y|^\alpha}\vphi(x)\,\dd \hn(x) \,=\,  \sum_{d=2}^{\infty} \sum_{i=1}^{\dim(\mathcal{H}^{N}_{d})} \mu_{d}^{N,\alpha} c_d^i H_{d}^{i}(y)\,,
$$
where the coefficient
\begin{equation} \label{eq:cinesi}
\mu_{d}^{N,\alpha}\,:=\,2^{N-1-\alpha}\frac{(N-1)\omega_{N-1}}{2}\biggl( \prod_{i=0}^{d-1}\Big(\frac{\alpha}{2}+i\Big) \biggr)\frac{\Gamma(\frac{N-1-\alpha}{2})\Gamma(\frac{N-1}{2})}
{\Gamma(N-1-\frac{\alpha}{2}+d)}
\end{equation}
is obtained by direct computation just integrating by parts. Therefore
$$
\int_{\partial B_{1}}\int_{\partial B_{1}} \frac{1}{|x-y|^\alpha}\vphi(x)\vphi(y)\,\dd \hn(x)\,\dd \hn(y) = \sum_{d=2}^{\infty} \sum_{i=1}^{\dim(\mathcal{H}^{N}_{d})} \mu_{d}^{N,\alpha}
(c_{d}^{i})^2.
$$
For the last term of \eqref{eq:secvarball}, noticing that the integral
$$
\mathcal{I}^{N,\alpha} \,:=\, \int_{B_{1}} \frac{\langle x-y, x\rangle}{|x-y|^{\alpha+2}} \,\dd y
$$
is independent of $x\in S^{N-1}$, we get
$$
\int_{\partial B_{1}}\Big( \int_{B_{1}} -\alpha\frac{\langle x-y, x\rangle}{|x-y|^{\alpha+2}} \,\dd y \Big)\vphi^2(x)\,\dd \hn(x) = -\alpha\mathcal{I}^{N,\alpha}\sum_{d=2}^{\infty}
\sum_{i=1}^{\dim(\mathcal{H}^{N}_{d})} (c_{d}^{i})^2.
$$
Combining all the previous equalities with \eqref{eq:secvarball} we obtain
$$
\partial^{2}\f(B_{R})[\tilde{\vphi}]  =  \sum_{d=2}^{\infty} \sum_{i=1}^{\dim(\mathcal{H}^{N}_{d})} R^{N-3}(c_{d}^{i})^2\Big[ d(d+N-2) - (N-1) + 2\gamma R^{N+1-\alpha}\Big(\mu_{d}^{N,\alpha} -
			\alpha\mathcal{I}^{N,\alpha}\Big)\Big].
$$

\subsection{Local minimality of the ball}

From the above expression we deduce that the quadratic form $\partial^{2}\f(B_{R})$ is strictly positive on $T^\bot(\partial B_R)$,
that is, the second variation of $\f$ at $B_R$ is positive according to Definition~\ref{def:varpos},
if and only if
\begin{equation} \label{eq:condamm}
d(d+N-2) - (N-1) + 2\gamma R^{N+1-\alpha}\Big(\mu_{d}^{N,\alpha} - \alpha\mathcal{I}^{N,\alpha}\Big) > 0
\end{equation}
for all $d\geq 2$,
where the ``only if'' part is due to the fact that $\mathcal{H}^{N}_{d}\subset T^\bot(S^{N-1})$ for each $d\geq2$.
On the contrary, $\partial^{2}\f(B_{R})[\tilde{\vphi}]<0$ for some $\tilde{\vphi}\in T^\bot(\partial B_R)$
if and only if there exists $d\geq2$ such that the left-hand side of \eqref{eq:condamm} is negative.

We want to write \eqref{eq:condamm} as a condition on $R$.
Since $d(d+N-2) - (N-1)>0$ for $d\geq 2$, we have that \eqref{eq:condamm} is certainly satisfied if $\mu_{d}^{N,\alpha} - \alpha\mathcal{I}^{N,\alpha}>0$.
But this is not always the case, as the following lemma shows.

\begin{lemma}
The sequence $\mu_{d}^{N,\alpha}$ strictly decreases to 0 as $d\rightarrow\infty$.
\end{lemma}

\begin{proof}
First of all we note that
\begin{equation}\label{eq:coeff}
\mu_{d+1}^{N,\alpha} = \frac{\frac{\alpha}{2}+d}{N-1-\frac{\alpha}{2}+d} \, \mu_{d}^{N,\alpha}\,,
\end{equation}
hence the sequence $(\mu_{d}^{N,\alpha})_{d\in\N}$ is decreasing since $\alpha<N-1$. Now
\begin{align*}
\mu_{d+1}^{N,\alpha} & = \Big( \prod_{k=1}^{d}\frac{\frac{\alpha}{2}+k}{N-1-\frac{\alpha}{2}+k} \Big)\mu_{1}^{N,\alpha}
=  \frac{\Gamma(N-\frac{\alpha}{2})\Gamma(1+\frac{\alpha}{2}+d)}{\Gamma(1+\frac{\alpha}{2})\Gamma(N-\frac{\alpha}{2}+d)} \mu_{1}^{N,\alpha}\\
& \sim_{d\rightarrow\infty} \frac{\Gamma(N-\frac{\alpha}{2})}{\Gamma(1+\frac{\alpha}{2})}\mu_{1}^{N,\alpha}
\sqrt{\frac{\frac{\alpha}{2}+d}{N-1-\frac{\alpha}{2}+d}}\,
\frac{e^{(\frac{\alpha}{2}+d)[\log(\frac{\alpha}{2}+d)-1]}}{e^{(N-1-\frac{\alpha}{2}+d) [\log( N-1-\frac{\alpha}{2}+d) -1]}}\,,
\end{align*}
where in the second equality we used the well known property $\Gamma(x+1)=x\Gamma(x)$,
and in the last step we used the Stirling's formula.
Since the previous quantity is infinitesimal as $d\to\infty$, we conclude the proof of the lemma.
\end{proof}

As a consequence of this lemma and of the fact that $\mathcal{I}^{N,\alpha}>0$, we have that the number
$$
d_{A}^{N,\alpha}:= \min\{ d\geq2 \,:\, \mu_{d}^{N,\alpha} < \alpha\mathcal{I}^{N,\alpha} \}
$$
\noindent
is well defined. This tells us that \eqref{eq:condamm} is satisfied for every $R>0$ if $d< d_A^{N,\alpha}$,
and for
$$
R < \biggl( \frac{d(d+N-2)-(N-1)}{2\gamma\big( \alpha\mathcal{I}^{N,\alpha}-\mu_{d}^{N,\alpha} \big)} \biggr)^{\frac{1}{N+1-\alpha}}=:g^{N,\alpha}(d).
$$
if $d\geq d_{A}^{N,\alpha}$.
Moreover, by the previous lemma we get that $g^{N,\alpha}(d)\rightarrow\infty$ as $d\rightarrow\infty$.
The following lemma tells us something more about the behaviour of the function $g^{N,\alpha}$.

\begin{lemma}
There exists a natural number $d_{I}^{N,\alpha}$ such that for $d<d_{I}^{N,\alpha}$ the function $g^{N,\alpha}$ is decreasing, while for $d>d_{I}^{N,\alpha}$ is increasing.
\end{lemma}

\begin{proof}
The condition $g^{N,\alpha}(d+1)>g^{N,\alpha}(d)$ is equivalent to
$$
\frac{(d+1)(d+1+N-2)-(N-1)}{2\gamma\big( \alpha\mathcal{I}^{N,\alpha}-\mu_{d+1}^{N,\alpha} \big)}>\frac{d(d+N-2)-(N-1)}{2\gamma\big( \alpha\mathcal{I}^{N,\alpha}-\mu_{d}^{N,\alpha} \big)}.
$$
Recalling \eqref{eq:coeff},
the above inequality can be rewritten, after some algebraic steps, as follows:
\begin{equation}\label{eq:condcresc}
\alpha\mathcal{I}^{N,\alpha}>\frac{d^2(N-\alpha+1)+d(N^2-\alpha N+\alpha-1)+\frac{\alpha}{2}(N-1)}{(N-1-\frac{\alpha}{2}+d)(2d+N-1)}\,\mu_{d}^{N,\alpha}.
\end{equation}
Using \eqref{eq:coeff}, it is easily seen that the right-hand side of the above inequality is decreasing
and converges to 0 as $d\to\infty$.
Hence the number
$$
d_{I}^{N,\alpha} := \min\{ d\in\N \,:\, \eqref{eq:condcresc} \mbox{ is satisfied} \}
$$
is well defined and satisfies the requirement of the lemma.
\end{proof}

We are now in position to prove Theorem~\ref{teo:minlocpalla}.

\begin{proof}[Proof of Theorem~\ref{teo:minlocpalla}]
Define
$$
\overline{R}(N,\alpha,\gamma):=\min_{d\geq d_{A}^{N,\alpha}}g^{N,\alpha}(d)\,,
$$
which can be characterized, by the previous lemmas, as
$$
\overline{R}(N,\alpha,\gamma):=\left\{
			\begin{array}{lr}
			 g^{N,\alpha}(d_{A}^{N,\alpha}) &\mbox{if } d_{A}^{N,\alpha}>d_{I}^{N,\alpha},\\
			 &\\
			 g^{N,\alpha}(d_{I}^{N,\alpha}) &\mbox{if } d_{A}^{N,\alpha}\leq d_{I}^{N,\alpha}.\\
			\end{array}
			\right.
$$
Now, from \eqref{eq:condamm}, we have that
$$
\partial^{2}\f(B_{R})[\tilde{\vphi}]>0 \text{ for every }\tilde{\vphi}\in T^\bot(\partial B_R)
\quad\Longleftrightarrow\quad R<\overline{R}(N,\alpha,\gamma),
$$
while
$$
\partial^{2}\f(B_{R})[\tilde{\vphi}]<0 \text{ for some }\tilde{\vphi}\in T^\bot(\partial B_R)
\quad\Longleftrightarrow\quad R>\overline{R}(N,\alpha,\gamma).
$$
By virtue of Theorem~\ref{teo:minl1} and Corollary~\ref{cor:minloc}, we obtain the first part of the theorem,
where $\mloc(N,\alpha,\gamma)$ is the volume of the ball of radius $\overline{R}(N,\alpha,\gamma)$.

In order to show that the critical radius tends to $\infty$ as $\alpha\to0$,
we notice that
$$
\partial^{2}\f(B_{R})[\tilde{\vphi}]\geq\sum_{d=2}^{\infty} \sum_{i=1}^{\dim(\mathcal{H}^{N}_{d})} (c_{d}^{i})^2R^{N-3}\bigl(N+1-2\gamma\alpha\mathcal{I}^{N,\alpha}R^{N+1-\alpha}\bigr).
$$
Since
$$
\mathcal{I}^{N,\alpha}\stackrel{\alpha\rightarrow0^+}{\longrightarrow}\int_{B_{1}} \frac{\langle x-y, x\rangle}{|x-y|^{2}} \,\dd y<\infty\,,
$$
we have that for each $R>0$ there exists $\overline{\alpha}(N,\gamma,R)>0$ such that for each $\alpha<\overline{\alpha}(N,\gamma,R)$
$$
\alpha\mathcal{I}^{N,\alpha}<\frac{N+1}{2\gamma R^{N+1-\alpha}},
$$
which immediately implies the claim.
To conclude the proof we examine in more details the special case $N=3$,
determining explicitly the critical mass $\mloc$.
From \eqref{eq:cinesi} we have that
\begin{eqnarray*}
\mu_{d}^{3,\alpha}
= 2^{2-\alpha}\pi\Big( \prod_{j=0}^{d-1}\Big( \frac{\alpha}{2}+j \Big) \Big)\frac{\Gamma(1-\frac{\alpha}{2})}{\Gamma(2+d-\frac{\alpha}{2})} 
= 2^{2-\alpha}\pi\alpha\frac{\Big( \prod_{j=1}^{d-1}\Big( \frac{\alpha}{2}+j \Big) \Big)}{\prod_{j=1}^{d-1}\Big(1-\frac{\alpha}{2}+j\Big)}\frac{1}{d+1-\frac{\alpha}{2}}\frac{1}{2-\alpha} \,,
\end{eqnarray*}
where we used the property $\Gamma(x+1)=x\Gamma(x)$.
Moreover, we compute explicitly in the Appendix the integral $\mathcal{I}^{3,\alpha}$, obtaining (see \eqref{eq:integral})
$$
\mathcal{I}^{3,\alpha} = 2\pi\frac{2^{2-\alpha}}{(4-\alpha)(2-\alpha)}.
$$
It is now easily seen that $d^{3,\alpha}_I=d^{3,\alpha}_A=2$ for every $\alpha\in(0,2)$.
Hence
$$
\overline{R}(3,\alpha,\gamma) = \biggl( \frac{(6-\alpha)(4-\alpha)}{2^{3-\alpha}\gamma\alpha\pi} \biggr)^{\frac{1}{4-\alpha}}\,,
$$
which completes the proof of the theorem.
\end{proof}


\section{Global minimality} \label{sect:global}

This section is devoted to the proof of the results concerning global minimality issues.
We start by showing how the information gained in Theorem~\ref{teo:minlocpalla}
can be used to prove the global minimality of the ball for small volumes.

\begin{proof}[Proof of Theorem~\ref{teo:minglobpalla}]
By scaling, we can equivalently prove that given $N\geq2$ and $\alpha\in(0,N-1)$ and setting
$$
\bar{\gamma} := \sup \bigl\{ \gamma>0 : B_1 \text{ is a global minimizer of } \f_{\alpha,\gamma} \text{ in }\R^N \text{ under volume constraint}\bigr\},
$$
we have that $\bar{\gamma}\in(0,\infty)$ and
$B_1$ is the unique global minimizer of $\f_{\alpha,\gamma}$ for every $\gamma<\bar{\gamma}$.

We start assuming by contradiction that there exist a sequence $\gamma_n\to0$
and a sequence of sets $E_n$, with $|E_n|=|B_1|$ and $\alpha(E_n,B_1)>0$, such that
\begin{equation} \label{eq:contraminglob}
\f_{\alpha,\gamma_n}(E_n)\leq\f_{\alpha,\gamma_n}(B_1).
\end{equation}
By translating $E_n$ so that $\alpha(E_n,B_1)=|E_n\triangle B_1|$, from \eqref{eq:contraminglob} one immediately gets
$$
C(N)\,|E_n\triangle B_1|^2
\leq \p(E_n)-\p(B_1)
\leq \gamma_n \bigl( \nl_\alpha(B_1)-\nl_\alpha(E_n) \bigr)
\leq \gamma_n c_0 |E_n\triangle B_1|
$$
where the first inequality follows from the quantitative isoperimetric inequality
and the last one from Proposition~\ref{prop:NLlipschitz}.
Hence, as $\gamma_n\to0$, we deduce that $\alpha(E_n,B_1)\to0$.

From the results of Section~\ref{sect:ball} it follows that
if $\gamma_0>0$ is sufficiently small then the functional $\f_{\alpha,\gamma_0}$ has positive second variation at $B_1$:
by Theorem~\ref{teo:minl1}, this implies the existence of a positive $\delta$ such that
\begin{equation}\label{eq:unifminglob}
\f_{\alpha,\gamma_0}(B_1)<\f_{\alpha,\gamma_0}(E)
\qquad\text{for every }E\text{ with }|E|=|B_1|\text{ and }0<\alpha(E,B_1)<\delta.
\end{equation}
We now want to show that \eqref{eq:unifminglob} holds for every $\gamma<\gamma_0$, with the same $\delta$.
Indeed, assuming by contradiction the existence of $\gamma<\gamma_0$
and $E\subset\R^N$ such that $|E|=|B_1|$, $0<\alpha(E,B_1)<\delta$ and
\begin{equation} \label{eq:unifminglob2}
\f_{\alpha,\gamma}(E)\leq\f_{\alpha,\gamma}(B_1),
\end{equation}
since $\p(B_1)<\p(E)$ we necessarily have $\nl_\alpha(E)<\nl_\alpha(B_1)$. Hence by \eqref{eq:unifminglob2}
\begin{equation}\label{eq:contr}
\p(E)-\p(B_1) \leq \gamma \bigl( \nl_\alpha(B_1)-\nl_\alpha(E) \bigr) < \gamma_0 \bigl( \nl_\alpha(B_1)-\nl_\alpha(E) \bigr)\,,
\end{equation}
that is, $\f_{\alpha,\gamma_0}(E)<\f_{\alpha,\gamma_0}(B_1)$, which contradicts \eqref{eq:unifminglob}.

Now, since for $n$ large enough we have that $\gamma_n<\gamma_0$ and $0<\alpha(E_n,B_1)<\delta$,
the previous property is in contradiction with \eqref{eq:contraminglob}.
This shows in particular that $\bar{\gamma}>0$.

The fact that $\bar{\gamma}$ is finite follows from Theorem~\ref{teo:minlocpalla},
which shows that for large masses the ball is not a local minimizer
(and obviously not even a global minimizer).

Finally, assume by contradiction that for some $\gamma<\bar{\gamma}$ the ball is not the unique global minimizer,
that is there exists a set $E$, with $|E|=|B_1|$ and $\alpha(E,B_1)>0$, such that $\f_{\alpha,\gamma}(E)\leq\f_{\alpha,\gamma}(B_1)$.
By definition of $\bar{\gamma}$, we can find $\gamma'\in(\gamma,\bar{\gamma})$ such that $B_1$ is a global minimizer of $\f_{\alpha,\gamma'}$.
Arguing as before, we have that by the isoperimetric inequality $\p(B_1)<\p(E)$,
which by our contradiction assumption implies that $\nl_\alpha(E)<\nl_\alpha(B_1)$;
this yields
$$
\p(E)-\p(B_1) \leq \gamma \bigl( \nl_\alpha(B_1)-\nl_\alpha(E) \bigr) < \gamma' \bigl( \nl_\alpha(B_1)-\nl_\alpha(E) \bigr)\,,
$$
which contradicts the fact that $B_1$ is a global minimizer for $\f_{\alpha,\gamma'}$.
\end{proof}

We now want to analyze what happens for small exponents $\alpha$.
Since for $\alpha=0$ the functional is just the perimeter, which is uniquely minimized by the ball,
the intuition suggests that the unique minimizer of $\f_{\alpha,\gamma}$, for $\alpha$ close to 0,
is the ball itself, as long as a minimizer exists.
In order to prove the theorem, we need an auxiliary result:
the non-existence volume threshold is uniformly bounded for $\alpha\in(0,1)$.
The proof is a simple adaptation of the argument contained in \cite[Section~2]{LuOtt},
where just the three-dimensional case with $\alpha=1$ is considered.

\begin{proposition}\label{prop:nonexist}
There exists $\bar{m}=\bar{m}(N,\gamma)<+\infty$ such that
for every $m>\bar{m}$ the minimum problem
\begin{equation} \label{minapiccolo}
I_m^\alpha:=\inf\left\{\f_{\alpha,\gamma}(E) \,:\, E\subset\R^N, \, |E|=m\right\}
\end{equation}
does not have a solution for every $\alpha\in(0,1)$.
\end{proposition}

\begin{proof}
During the proof we will denote by $C$ a generic constant, depending only on $N$ and $\gamma$,
which may change from line to line.

\smallskip
\noindent
{\it Step 1.}
We claim that there exists a constant $C_0$, depending only on $N$ and $\gamma$, such that
\begin{equation} \label{claim1apiccolo}
I_m^\alpha \leq C_0\,m
\qquad
\text{for every }0<\alpha<N-1\text{ and }m\geq1.
\end{equation}
Indeed, if $B$ is a ball of volume $m$, then
$$
\f_{\alpha,\gamma}(B) = N{\omega_N}^{1/N}m^{{(N-1)}/{N}}
+ \gamma \, c_\alpha \left(\frac{m}{\omega_N}\right)^{\frac{2N-\alpha}{N}},
\qquad
c_\alpha := \int_{B_1}\int_{B_1}\frac{1}{|x-y|^\alpha}\,\dd x\dd y\,.
$$
It follows that for every $1\leq m<2$ we have $I_m^\alpha\leq C_0$, for some constant $C_0$ depending only on $N$ and $\gamma$.
It is now easily seen that $I_m^\alpha\leq I_{m_1}^\alpha + I_{m_2}^\alpha$ if $m=m_1+m_2$
(see the proof of \cite[Lemma~3]{LuOtt}): hence by induction on $k$
we obtain $I_m^\alpha\leq C_0k$ for every $m\in[k,k+1)$.

\smallskip
\noindent
{\it Step 2.}
We claim that there exists a constant $C_1$, depending only on $N$ and $\gamma$, such that
for every $0<\alpha<N-1$ and $m\geq1$, if $E$ is a solution to \eqref{minapiccolo} then
\begin{equation} \label{claim2apiccolo}
|E\cap B_R(x)|\geq C_1 R^N
\end{equation}
for every $R\leq1$ and for every $x\in E$ such that $|E\cap B_r(x)|>0$ for all $r>0$.

To prove the claim, assume without loss of generality that $x=0$.
It is clearly sufficient to show \eqref{claim2apiccolo} for $\mathcal{L}^1$-a.e. $R<\e_0$, where $\e_0$ will be fixed later in the proof.
In particular, from now on we can assume without loss of generality that $R$ is such that $\mathcal{H}^{N-1}(\partial E\cap\partial B_R)=0$.
We compare the energies of $E$ and $E':=\lambda(E\setminus B_R)$, where $\lambda>1$ is such that $|E'|=m$:
by minimality of $E$ we have $\f_{\alpha,\gamma}(E)\leq\f_{\alpha,\gamma}(E')$,
which gives after a direct computation
$$
\hn(\partial E\cap B_R) \leq (\lambda^{2N-\alpha}-1)\f_{\alpha,\gamma}(E) + \lambda^{N-1}\hn(\partial B_R\cap E).
$$
In turn this implies, by using $\hn(\partial(E\cap B_R))=\hn(\partial E\cap B_R)+\hn(\partial B_R\cap E)$
(recall that $\mathcal{H}^{N-1}(\partial E\cap\partial B_R)=0$),
$$
\hn(\partial(E\cap B_R)) \leq (\lambda^{2N-\alpha}-1)\f_{\alpha,\gamma}(E) + (\lambda^{N-1}+1)\hn(\partial B_R\cap E).
$$
Now, choosing $\e_0>0$ so small that $|E\setmeno B_R|\geq \frac12 m$, we obtain the following estimates:
$$
\lambda^{2N-\alpha}-1 = \Biggl( \frac{m}{|E\setmeno B_R|} \Biggr)^{\frac{2N-\alpha}{N}} - 1
\leq C\Biggl( \frac{m}{|E\setmeno B_R|} -1 \Biggr)
\leq C \frac{|E\cap B_R|}{m}\,,
\quad
\lambda^{N-1}\leq C\,.
$$
Hence from the isoperimetric inequality, \eqref{claim1apiccolo}, and from the above estimates we deduce that
\begin{align*}
|E\cap B_R|^{\frac{N-1}{N}} \leq C |E\cap B_R| + C \hn(\partial B_R\cap E).
\end{align*}
Finally, observe that if $\e_0$ is sufficiently small we also have $|E\cap B_R|\leq \frac{1}{2C}|E\cap B_R|^{\frac{N-1}{N}}$,
hence we obtain
$$
|E\cap B_R|^{\frac{N-1}{N}} \leq C\hn(\partial B_R\cap E) = C \frac{\dd}{\dd R}|E\cap B_R|,
$$
which yields
$$
\frac{\dd}{\dd R}|E\cap B_R|^{\frac1N}\geq C
\qquad\text{for }\mathcal{L}^1\text{-a.e. }R<\e_0.
$$
By integrating the previous inequality we conclude the proof of the claim.

\smallskip
\noindent
{\it Step 3.}
We claim that there exists a constant $C_2$, depending only on $N$ and $\gamma$, such that
for every $0<\alpha<1$ and $m\geq1$, if $E$ is a solution to \eqref{minapiccolo} then
\begin{equation} \label{claim3apiccolo}
\nl_\alpha(E)\geq C_2\,m\log m - C_2\,m
\end{equation}
(notice that the conclusion of the proposition follows immediately from \eqref{claim1apiccolo} and \eqref{claim3apiccolo}).

In order to prove the claim, we first observe that
\begin{equation}\label{claim4apiccolo}
|E\cap B_R(x)|\geq CR
\qquad\text{for every }x\in E\text{ and }1<R<\textstyle{\frac12}{\rm diam}(E).
\end{equation}
Indeed, as $E$ is not contained in $B_R(x)$ and $E$ is connected (see Theorem~\ref{teo:regmin}),
we can find points $x_i\in E\cap\partial B_{R-i}(x)$ for $i=1,\ldots,\lfloor R \rfloor$
such that $|E\cap B_r(x_i)|>0$ for every $r>0$. Then by \eqref{claim2apiccolo}
$$
|E\cap B_R(x)| \geq \sum_{i=1}^{\lfloor R \rfloor} |E\cap B_{\frac12}(x_i)| \geq C_1 \left(\frac12\right)^N\lfloor R \rfloor\,.
$$

Observe now that, if we set $E_R:=\{(x,y)\in E\times E : |x-y|<R\}$,
we have by \eqref{claim4apiccolo} that for every $1<R<\frac12{\rm diam}(E)$
\begin{align}\label{claim5apiccolo}
\LL^{2N}(E_R)
= \int_{E} |E\cap B_R(x)|\,\dd x
\geq C |E| R.
\end{align}
Hence
\begin{align*}
\nl_\alpha(E)
&   = \int_E\int_E\frac{1}{|x-y|^\alpha}\,\dd x\dd y
    = \frac{1}{\sqrt{2}}\int_0^{+\infty} \frac{1}{R^\alpha} \, \mathcal{H}^{2N-1}(\partial E_R)\,\dd R \\
&   = \frac12 \int_0^{+\infty} \frac{1}{R^\alpha} \, \frac{\dd}{\dd R}\LL^{2N}(E_R)\,\dd R
    \geq \frac12 \int_1^{+\infty} \frac{1}{R} \, \frac{\dd}{\dd R}\LL^{2N}(E_R)\,\dd R \\
&   = - \frac12\,\LL^{2N}(E_1) + \frac12\int_1^{+\infty} \frac{1}{R^2} \, \LL^{2N}(E_R)\,\dd R \\
&   \geq -Cm + Cm\int_{1}^{\frac12{\rm diam}(E)} \frac{1}{R}\,\dd R\,,
\end{align*}
where in the first inequality we used the fact that $\alpha<1$,
while the second one follows from \eqref{claim5apiccolo}.
This completes the proof of the proposition.
\end{proof}

An essential remark for the proof of Theorem~\ref{teo:apiccolo} is contained in the following lemma,
where it is shown that the local minimality neighbourhood of the ball is in fact \emph{uniform}
with respect to $\gamma$ and $\alpha$.

\begin{lemma} \label{lemma:uniformity}
Given $\bar{\gamma}>0$, there exist $\bar{\alpha}>0$ and $\delta>0$ such that
$$
\f_{\alpha,\gamma}(B_1) < \f_{\alpha,\gamma}(E)
$$
for every $\alpha\leq\bar{\alpha}$, for every $\gamma\leq\bar{\gamma}$
and for every set $E$ with $|E|=|B_1|$ and $0<\alpha(E,B_1)<\delta$.
\end{lemma}

\begin{proof}[Proof (sketch)]
Notice that, by the same argument used in the proof of Theorem~\ref{teo:minglobpalla},
it is sufficient to show that, given $\bar{\gamma}>0$, there exist $\bar{\alpha}>0$ and $\delta>0$ such that
$$
\f_{\alpha,\bar{\gamma}}(B_1) < \f_{\alpha,\bar{\gamma}}(E)
$$
for every $\alpha\leq\bar{\alpha}$ and for every set $E$ with $|E|=|B_1|$ and $0<\alpha(E,B_1)<\delta$.

\smallskip
In order to prove this property, we start by observing that there exists $\alpha_1>0$ such that
\begin{equation} \label{unifminglob3}
m_0 := \inf_{\alpha\leq\alpha_1} \inf\left\{
\partial^2\f_{\alpha,\bar{\gamma}}(B_1)[\vphi] \, : \, \vphi\in T^\perp(\partial B_1), \, \|\vphi\|_{\widetilde{H}^1(\partial B_1)}=1 \right\} >0\,.
\end{equation}
In fact, assuming by contradiction the existence of $\alpha_n\to0$ and $\vphi_n\in T^\bot(\partial B_1)$,
with $\|\vphi_n\|_{\widetilde{H}^1}=1$, such that $\partial^2\f_{\alpha_n,\bar{\gamma}}(B_1)[\vphi_n]\to0$,
we have by compactness that, up to subsequences, $\vphi_n\wto\vphi_0$ weakly in $H^1$ for some $\vphi_0\in T^\bot(\partial B_1)$.
It is now not hard to show that the last two integrals in the quadratic form $\partial^2\f_{\alpha_n,\bar{\gamma}}(B_1)[\vphi_n]$ converge to 0 as $n\to\infty$:
indeed, the second integral in \eqref{eq:fquad} converges to $0$, since it is equal to
$$
-\alpha_n\int_{\partial B_1} \biggl( \int_{B_1}\frac{x-y}{|x-y|^{\alpha_n+2}}\,\dd y \biggr) \cdot x \, \vphi_n^2(x)\,\dd \hn(x)
\leq C\alpha_n\int_{\partial B_1} \vphi_n^2\,\dd\hn
\rightarrow 0\qquad\mbox{as }n\to\infty\,.
$$
For the last integral in \eqref{eq:fquad}, denoting by
$G_{\alpha_n}(x,y):=|x-y|^{-\alpha_n}$, we write
\begin{align*}
\int_{\partial B_1}&\int_{\partial B_1}G_{\alpha_n}(x,y) \vphi_n(x) \vphi_n(y)\,\dd\hn(x)\dd\hn(y) \\
&=	\int_{\partial B_1}\int_{\partial B_1}G_{\alpha_n}(x,y) \vphi_n(x) (\vphi_n(y)-\vphi_0(y))\,\dd\hn(x)\dd\hn(y) \\
&+ \int_{\partial B_1}\int_{\partial B_1}G_{\alpha_n}(x,y) (\vphi_n(x)-\vphi_0(x)) \vphi_0(y)\,\dd\hn(x)\dd\hn(y) \\
&+ \int_{\partial B_1}\int_{\partial B_1}G_{\alpha_n}(x,y) \vphi_0(x) \vphi_0(y)\,\dd\hn(x)\dd\hn(y)\,;
\end{align*}
the potential estimates provided by Lemma~\ref{lemma:potenziale}, where the constant can be chosen independently of $\alpha_n$ by Remark~\ref{rm:potenziale}, guarantee that the first two integrals in the above expression converge to zero,
while also the third one vanishes in the limit by the Lebesgue's Dominate Convergence Theorem,
recalling that $\int_{\partial B_1}\vphi_0=0$ and $\alpha_n\to0$.
Moreover, for the first integral in the quadratic form $\partial^2\f_{\alpha_n,\bar{\gamma}}(B_1)[\vphi_n]$, we have that
$$
\int_{\partial B_1}|D_\tau\vphi_0|^2 \leq \liminf_{n\to\infty} \int_{\partial B_1}|D_\tau\vphi_n|^2 \,,
\qquad
\int_{\partial B_1}|B_{\partial B_1}|^2\vphi_n^2 \to \int_{\partial B_1}|B_{\partial B_1}|^2\vphi_0^2\,.
$$
Hence, if $\vphi_0=0$ we conclude that $\int_{\partial B_1}|D_\tau\vphi_n|^2\to0$, which contradicts the fact that $\|\vphi_n\|_{\widetilde{H}^1}=1$ for every $n$.
On the other hand, if $\vphi_0\neq0$, we obtain
$$
\int_{\partial B_1}|D_\tau\vphi_0|^2\,\dd\hn - \int_{\partial B_1}|B_{\partial B_1}|^2\vphi_0^2\,\dd\hn\leq0\,,
$$
that is, the second variation of the area functional computed at the ball $B_1$ is not strictly positive, which is again a contradiction.

\smallskip
With condition \eqref{unifminglob3}, it is straightforward to check that the proof of Theorem~\ref{teo:minw} provides the existence
of $\delta_1>0$ and $C_1>0$ such that
$$
\f_{\alpha,\bar{\gamma}}(E) \geq \f_{\alpha,\bar{\gamma}}(B_1) + C_1\bigl( \alpha(E,B_1) \bigr)^2
$$
for every $\alpha\leq\alpha_1$ and for every $E\subset\R^N$ with $|E|=|B_1|$ and $\partial E=\{x+\psi(x)x:x\in\partial B_1\}$
for some $\psi$ with $\|\psi\|_{W^{2,p}(\partial B_1)}<\delta_1$.

\smallskip
In turn, having proved this property one can repeat the proofs of Theorem~\ref{teo:mininf} and Theorem~\ref{teo:minl1}
to deduce that there exist $\alpha_2>0$, $\delta_2>0$ and $C_2>0$ such that
$$
\f_{\alpha,\bar{\gamma}}(E) \geq \f_{\alpha,\bar{\gamma}}(B_1) + C_2\bigl( \alpha(E,B_1) \bigr)^2
$$
for every $\alpha\leq\alpha_2$ and for every $E\subset\R^N$ with $|E|=|B_1|$ and $\alpha(E,B_1)<\delta_2$.
The only small modifications consist in assuming, in the contradiction arguments, also the existence of sequences $\alpha_n\to0$,
instead of working with a fixed $\alpha$.
Then the essential remark is that the constant $c_0$ provided by Proposition~\ref{prop:NLlipschitz} is independent of $\alpha_n$.
In addition, some small changes are required in the last part of the proof,
since the functions $v_{F_n}$ associated, according to \eqref{eq:ve}, with the sets $F_n$ constructed in the proof
are defined with respect to different exponents $\alpha_n$,
but observe that the bounds provided by Proposition~\ref{prop:ve} are still uniform.
The easy details are left to the reader.

These observations complete the proof of the lemma.
\end{proof}

We are now in position to complete the proof of Theorem~\ref{teo:apiccolo}.

\begin{proof}[Proof of Theorem~\ref{teo:apiccolo}]
We assume by contradiction that there exist $\alpha_n\to0$, $m_n>0$ and sets $E_n\subset\R^N$,
with $|E_n|=m_n$, $\alpha(E_n,B_n)>0$ (where we denote by $B_n$ a ball with volume $m_n$),
such that $E_n$ is a global minimizer of $\f_{\alpha_n,\gamma}$ under volume constraint.
Note that, as the non-existence threshold is uniformly bounded for $\alpha\in(0,1)$ (Proposition~\ref{prop:nonexist}),
we can assume without loss of generality that $m_n\leq\bar{m}<+\infty$.

By scaling, we can rephrase our contradiction assumption as follows: there exist
$\alpha_n\to0$, $\gamma_n>0$ with $\bar{\gamma}:=\sup_n\gamma_n<+\infty$,
and $F_n\subset\R^N$ with $|F_n|=|B_1|$, $\alpha(F_n,B_1)>0$ such that
$$
\f_{\alpha_n,\gamma_n}(F_n) = \min\{\f_{\alpha_n,\gamma_n}(F) \,:\, |F|=|B_1|\}\,,
$$
and in particular
\begin{equation}\label{contraapiccolo}
\f_{\alpha_n,\gamma_n}(F_n)\leq\f_{\alpha_n,\gamma_n}(B_1).
\end{equation}
We now claim that, since $\alpha_n\to0$,
\begin{equation}\label{claim6apiccolo}
\lim_{n\to+\infty}|\nl_{\alpha_n}(B_1)-\nl_{\alpha_n}(F_n)|=0.
\end{equation}
Indeed, we observe that by adapting the first step of the proof of Theorem~\ref{teo:regmin},
we have that there exists $\Lambda>0$ (independent of $n$) such that $F_n$ is also a solution to the penalized minimum problem
$$
\min\left\{ \f_{\alpha_n,\gamma_n}(F) + \Lambda\big| |F|-|B_1| \big| \,:\, F\subset\R^N \right\}
$$
(for $n$ large enough).
In turn, this implies that each set $F_n$ is an $(\omega,r_0)$-minimizer for the area functional (see Definition~\ref{def:quamin})
for some positive $\omega$ and $r_0$ (independent of $n$): in fact for every finite perimeter set $F$ with $F\triangle F_n\subset\subset B_{r_0}(x)$ we have by minimality of $F_n$
\begin{align*}
\p(F_n)
&\leq \p(F) + \gamma_n \bigl( \nl_{\alpha_n}(F)-\nl_{\alpha_n}(F_n)\bigr) + \Lambda \big||F|-|B_1|\big|\\
&\leq \p(F) + \bigl( \bar{\gamma} c_0 + \Lambda \bigr) |F\triangle F_n|,
\end{align*}
where we used Proposition~\ref{prop:NLlipschitz} and the fact that the constant $c_0$ can be chosen independently of $\alpha_n$.
We can now use the uniform density estimates for $(\omega,r_0)$-minimizers (see \cite[Theorem~21.11]{Mag}),
combined with the connectedness of the sets $F_n$ (see Theorem~\ref{teo:regmin}),
to deduce that (up to translations) they are equibounded: there exists $\bar{R}>0$ such that $F_n\subset B_{\bar{R}}$ for every $n$.
Using this information, it is now easily seen that, since $\alpha_n\to0$,
$$
\nl_{\alpha_n}(F_n)=\int_{F_n}\int_{F_n}\frac{1}{|x-y|^{\alpha_n}}\,\dd x\dd y \to |B_1|^2\,,
$$
from which \eqref{claim6apiccolo} follows.

By \eqref{contraapiccolo}, \eqref{claim6apiccolo} and using the quantitative isoperimetric inequality we finally deduce
\begin{align*}
C_N \bigl(\alpha(F_n,B_1)\bigr)^2
&   \leq \p(F_n)-\p(B_1)
    \leq \gamma_n \bigl( \nl_{\alpha_n}(B_1)-\nl_{\alpha_n}(F_n) \bigr)\\
&   \leq \bar{\gamma} \, \big| \nl_{\alpha_n}(B_1)-\nl_{\alpha_n}(F_n) \big|
    \to0,
\end{align*}
that is, $F_n$ converges to $B_1$ in $L^1$.
Hence \eqref{contraapiccolo} is in contradiction with Lemma~\ref{lemma:uniformity}
for $n$ large enough.
\end{proof}

We conclude this section with the proof of Theorem~\ref{teo:apiccolo2}.

\begin{proof}[Proof of Theorem~\ref{teo:apiccolo2}]
First of all we notice
that, since for masses smaller than $\mglob$ the ball is the unique global minimizer,
for each $m>0$ there exists $k_m\in\N$ such that $f_{k_m}(m)=\min_if_i(m)$.
Setting $m_0=0$, $m_1=\mglob$, we have by Theorem~\ref{teo:apiccolo} that \eqref{eq:apiccolo2} holds for $k=1$.
In the following, we denote by $E^m_R$ a solution to the constrained minimum problem
$$
\min\{ \f(E): E\subset B_{R},\, |E|=m \}.
$$
We remark that
\begin{equation}\label{eq:convinf}
\f(E^m_R)\rightarrow \inf \left\{\f(E): E\subset\R^N,\, |E|=m \right\}\, \mbox{ as } R\rightarrow\infty\,,
\end{equation}
and that, given $\bar{m}>0$, for every $m<\bar{m}$ and for every $R>0$ the volume of each connected component of $E^m_R$
is bounded from below by a positive constant $M_{\bar{m}}>0$ depending on $\bar{m}$
(this conclusion can be obtained by arguing as in the proof of Theorem~\ref{teo:regmin},
showing in particular that each set $E^m_R$ is an $(\omega, r_0)$-minimizer for some constant $\omega$ independent of $m\leq\bar{m}$).

We now define
$$
m_2:=\sup\bigl\{m\geq m_1 : f_2(m')={\textstyle \inf_{|E|=m'}}\f(E) \text{ for each }m'\in[m_1,m)\bigr\}
$$
and we show that $m_2>m_1$.
Indeed, fix $\e>0$ and $m\in(m_1,m_1+\e)$.
Observe that the sets $(E^m_R)_R$ cannot be equibounded,
or otherwise they would converge (as $R\to\infty$) to a global minimizer of $\f$ with volume $m$,
whose existence is excluded by Theorem~\ref{teo:apiccolo}.
The fact that the diameter of $E^m_R$ tends to infinity,
combined with the uniform density lower bound satisfied by $E^m_R$ (which, in turn, follows from the quasiminimality property),
guarantees that for all $R$ large enough the set $E_R^m$ is not connected;
moreover, if $\e$ is small enough, each of its connected component has mass smaller than $\mglob$,
as a consequence of the lower bound on the volume of the connected components.
Then we can write $E^m_R = F_1\cup F_2$, with $|F_1|,|F_2|<\mglob$ and $F_1\cap F_2=\emptyset$,
so that we can decrease the energy of $E^m_R$ by replacing each $F_i$ by a ball of the same volume,
sufficiently far apart from each other, obtaining that $f_2(m)\leq\f(E^m_R)$.
By \eqref{eq:convinf} we easily conclude that
$f_2(m)= \inf_{|E|=m}\f(E)$ for every $m\in(m_1,m_1+\e)$, from which follows that $m_2>m_1$.
Moreover, by definition of $m_2$, we have that \eqref{eq:apiccolo2} holds for $k=2$.

We now proceed by induction, defining
$$
m_{k+1}:=\sup\bigl\{m\geq m_k \,:\, f_{k+1}(m') ={\textstyle \inf_{|E|=m'}}\f(E) \text{ for each }m'\in[m_k,m)\bigr\}
$$
and showing that $m_k<m_{k+1}$.
Arguing as before, we consider $m\in(m_{k},m_{k}+\e)$, for some $\e>0$ small enough,
and we observe that for $R$ sufficiently large the set $E_R^m$ is not connected,
and each of its connected components has volume belonging to an interval $(m_{i-1},m_i]$ for some $i\leq k$.
By the inductive hypothesis we can obtain a new set $F_R^m$, union of a finite number of disjoint balls,
such that $\f(F_R^m)\leq\f(E_R^m)$, simply by replacing each connected component of $E_R^m$ by a disjoint union of balls.
We can also assume that at least one of these balls, say $B$, has volume larger than $\e$ (if we choose for instance $\e<\frac{m_1}{2}$);
in this way $|F_R^m\setminus B|<m_k$ and we can decrease the energy of $F_R^m$ by replacing $F_R^m\setminus B$ by a finite union of at most $k$ balls.
With this procedure we find a disjoint union of at most $k+1$ balls whose energy is smaller than $\f(F_R^m)$,
so that, recalling \eqref{eq:convinf} and that $\f(F_R^m)\leq\f(E_R^m)$,
we conclude that $f_{k+1}(m)=\inf_{|E|=m}\f(E)$ for every $m\in(m_k,m_k+\e)$.
This completes the proof of the inequality $m_k<m_{k+1}$,
and shows also, by definition of $m_{k}$, that \eqref{eq:apiccolo2} holds.

Now, assume by contradiction that $m_k\to\bar{m}<\infty$ as $k\to\infty$.
Since each interval $(m_k,m_{k+1})$ is not degenerate, the definition of $m_k$ as a supremum
ensures that we can find an increasing sequence of masses $\bar{m}_k\to\bar{m}$
such that an optimal configuration for $\min_if_i(\bar{m}_k)$ is given by exactly $k+1$ balls.
Recalling that the constant $M_{\bar{m}}$ provides a lower bound on the volume of each ball of an optimal configuration,
the previous assertion is impossible for $k$ large and shows that $\lim_{k\to\infty}m_k=\infty$.
Finally, it is clear that the number of non-degenerate balls tends to $\infty$ as $m\to\infty$,
since the volume of each ball in an optimal configuration for $\min_if_i(m)$ must be not larger than $m_1$.
\end{proof}

\section{Appendix} \label{sect:app}

\subsection{Computation of the first and second variations}

\begin{proof}[Proof of Theorem~\ref{teo:var}]
The first and the second variations of the perimeter of a regular set $E$ are standard calculations
(see, \emph{e.g.}, \cite{Sim}) and lead to
\begin{equation}\label{eq:firvarper}
\frac{\dd}{\dd t}{\p(E_t)}_{|_{t=0}} = \int_{\partial E} H_{\partial E}\langle X,\nu_{E}\rangle \,\hn
\end{equation}
and
\begin{align}\label{eq:secvarper}
\frac{\dd^2}{\dd t^2}{\p(E_t)}_{|_{t=0}} &= \int_{\partial E}\left( \,|D_{\tau}\langle X,\nu_{E}\rangle|^2-|B_{\partial E}|^2\langle X,\nu_{E}\rangle^2 \right)\,\dd\hn \nonumber\\
& \hspace{.5cm}+ \int_{\partial E} H_{\partial E} \left( \langle X,\nu_E \rangle \,\div X - \div_\tau\bigl(X_\tau\langle X,\nu_E\rangle \bigr)\right)\dd\hn\,.
\end{align}
This particular form of the second variation is in fact obtained in \cite[Proposition~3.9]{CagMorMor},
and we rewrote the last term according to \cite[equation~(7.5)]{AceFusMor}.

So now on we will focus on the calculation of the first and the second variation of the nonlocal part. In order to compute these quantities we introduce the smoothed potential
$$
G_{\delta}(a,b):=\frac{1}{(|a-b|^2+\delta^2)^{\frac{\alpha}{2}}}
$$
for $\delta>0$, and the associated nonlocal energy
$$
\nl_{\delta}(F):=\int_{F}\int_{F}G_{\delta}(a,b)\,\dd a\dd b.
$$
We remark that the following identities hold:
\begin{equation}\label{eq:idnabla}
\nabla_x \bigl( G_{\delta}(\Phi_t(x),\Phi_t(y)) \bigr) = \nabla_a G_{\delta}(\Phi_t(x),\Phi_t(y))\cdot D\Phi_t(x),
\end{equation}
\begin{equation}\label{eq:idnabla2}
\nabla_b G_{\delta}(a,b) = \nabla_a G_{\delta}(b,a).
\end{equation}

\smallskip
\noindent
{\it Step 1: first variation of the nonlocal term.}
The idea to compute the first variation of the nonlocal part is to prove the following two steps:
\begin{enumerate}
\item $\nl_{\delta}(E_t)\stackrel{\delta\rightarrow0}{\longrightarrow}\nl(E_t)$ uniformly for $t\in(-t_0,t_0)$,
\item $\frac{\partial}{\partial t}\nl_{\delta}(E_t)$ converges uniformly for $t\in(-t_0,t_0)$ to some function $H(t)$ as $\delta\to0$,
\end{enumerate}
where $t_0<1$ is a fixed number. From (1) and (2) it follows that
\begin{equation} \label{var1}
\frac{\dd}{\dd t}{\nl(E_t)}_{|_{t=0}}=H(0)=\lim_{\delta\rightarrow0}\frac{\partial}{\partial t}{\nl_{\delta}(E_t)}_{|_{t=0}}.
\end{equation}

We prove (1). We have that
\begin{eqnarray*}
|\nl_{\delta}(E_t)-\nl(E_t)|
= \bigg| \int_{E_t}\int_{E_t} \bigl(G_{\delta}(x,y)-G(x,y)\bigr)\,\dd x\dd y \bigg|
\leq \int_{B_R}\int_{B_R} |G_{\delta}(x,y)-G(x,y)|\,\dd x\dd y\,,
\end{eqnarray*}
where we have used the fact that $E$ is bounded and hence $E_t\subset B_R$ for some ball $B_R$.
It is now easily seen that the last integral in the previous expression tends to 0 as $\delta\to0$,
thanks to the Lebesgue's Dominated Convergence Theorem, hence
$$
\sup_{t\in(-t_0,t_0)}|\nl_{\delta}(E_t)-\nl(E_t)|\rightarrow0
\qquad\text{as }\delta\to0.
$$

We now prove (2). By a change of variables and using \eqref{eq:idnabla} and \eqref{eq:idnabla2} we have
\begin{align*}
\frac{\partial}{\partial t}\nl_{\delta}(E_t)
&= 2\int_{E}\int_{E}\frac{\partial J\Phi_t}{\partial t}(x)J\Phi_t(y)G_{\delta}(\Phi_t(x),\Phi_t(y))\,\dd x\dd y \\
&\hspace{.3cm}+ 2\int_{E}\int_{E}J\Phi_t(x)J\Phi_t(y)\langle\nabla_x \bigl( G_{\delta}(\Phi_t(x),\Phi_t(y)) \bigr) \cdot (D\Phi_t(x))^{-1}, X(\Phi_t(x))\rangle\,\dd x\dd y \\
&= \int_{E}\int_{E}f(t,x,y)G_{\delta}(\Phi_t(x),\Phi_t(y))\,\dd x\dd y \\
&\hspace{.3cm}+ \int_{\partial E} \left(\int_{E} g(t,x,y)G_{\delta}(\Phi_t(x),\Phi_t(y))\,\dd y \right)\,\dd \hn(x) \,,
\end{align*}
where $J\Phi_t := \det (D\Phi_t)$ is the jacobian of the map $\Phi_t$,
$$
f(t,x,y):=2\frac{\partial J\Phi_t}{\partial t}(x)J\Phi_t(y)-2\div_x\bigl(J\Phi_t(x)J\Phi_t(y)X(\Phi_t(x))\cdot(D\Phi_t(x))^{-T}\bigr)\,,
$$
$$
g(t,x,y):=J\Phi_t(x)J\Phi_t(y)\langle X(\Phi_t(x))\cdot(D\Phi_t(x))^{-T},\nu(x) \rangle
$$
and in the last step we used integration by parts and Fubini's Theorem.
Now since $f$ and $g$ are uniformly bounded on $(-t_0,t_0)\times E\times E$ and $(-t_0,t_0)\times \partial E\times E$ respectively,
it is easily seen that
$$
\frac{\partial}{\partial t}\nl_{\delta}(E_t)\stackrel{\delta\rightarrow0}{\longrightarrow} H(t)\,\;\,\mbox{uniformly for }
t\in(-t_0,t_0)\,,
$$
where
$$
H(t):=\int_{E}\int_{E}f(t,x,y)G(\Phi_t(x),\Phi_t(y))\,\dd x\dd y + \int_{\partial E} \left(\int_{E} g(t,x,y)G(\Phi_t(x),\Phi_t(y))\,\dd y \right)\,\dd \hn(x).
$$

We finally compute \eqref{var1}.
Recalling that
\begin{equation} \label{eq:jac}
{\frac{\partial J\Phi_t}{\partial t}}_{|_{t=0}}=\div X\,,
\end{equation}
we have
\begin{eqnarray*}
\frac{\partial}{\partial t}{\nl_{\delta}(E_t)}_{|_{t=0}} & = & 2\int_{E}\int_{E}\biggl( \frac{\div X(x)}{(|x-y|^2+\delta^2)^{\frac{\alpha}{2}}} -\alpha
																											\frac{\langle X(x),x-y \rangle}{(|x-y|^2+\delta^2)^{\frac{\alpha+2}{2}}} \biggr)\,\dd x\dd y  \\
																							 & = & 2\int_{E}\int_{E} \div_x\left( \frac{X(x)}{(|x-y|^2+\delta^2)^{\frac{\alpha}{2}}} \right)\,\dd x\dd y  \\
																							 & = & 2\int_{\partial E}\left( \int_{E} \frac{\langle X(x), \nu(x) \rangle}{(|x-y|^2+\delta^2)^{\frac{\alpha}{2}}} \,\dd y\right)\,\dd \hn(x)
\end{eqnarray*}
(where we used the divergence Theorem and Fubini's Theorem in the last equality),
and hence by letting $\delta\to0$ we conclude that
\begin{eqnarray*}
H(0)
= 2\int_{\partial E}\left( \int_{E} \frac{\langle X(x), \nu(x) \rangle}{|x-y|^{\alpha}} \,\dd y \right)\,\dd\hn(x)
= 2\int_{\partial E} v_E \,\langle X,\nu \rangle\,\dd\hn\,.
\end{eqnarray*}
This, combined with \eqref{eq:firvarper}, concludes the proof of the formula for the first variation of $\f$.

\smallskip
\noindent
{\it Step 2: second variation of the nonlocal term.}
We will compute the second variation of the nonlocal term by showing that
$$
\frac{\partial^2}{\partial t^2}\nl_{\delta}(E_{t})\stackrel{\delta\rightarrow0}{\longrightarrow}K(t)\;\;\mbox{uniformly in}\;t\in(-t_0,t_0)
$$
for some function $K$, hence getting
\begin{equation}\label{var2}
\frac{\dd^2}{\dd t^2}{\nl(E_{t})}_{|_{t=0}} = K(0) = \lim_{\delta\rightarrow0}\frac{\partial^2}{\partial t^2}{\nl_{\delta}(E_{t})}_{|_{t=0}}\,.
\end{equation}
First of all we have that
\begin{align} \label{var3}
\frac{\partial^2}{\partial t^2} \nl_{\delta}(E_{t})
&   = \frac{\partial}{\partial t} \biggl[2\int_{E}\int_{E}\frac{\partial J\Phi_t}{\partial t}(x) J\Phi_t(y) G_{\delta}(\Phi_t(x),\Phi_t(y)) \,\dd x\dd y \nonumber\\
&   \hspace{0.5cm} + 2\int_{E}\int_{E}J\Phi_t(x)J\Phi_t(y)\langle\nabla_a G_{\delta}(\Phi_t(x),\Phi_t(y)), X(\Phi_t(x))\rangle\,\dd x\dd y \biggr] \nonumber\\
&   = 2\int_{E}\int_{E}\frac{\partial}{\partial t}\Big(\frac{\partial J\Phi_t}{\partial t}(x)J\Phi_t(y)\Big) G_{\delta}(\Phi_t(x),\Phi_t(y)) \,\dd x\dd y \nonumber\\
&   \hspace{0.5cm}+ 2\int_{E}\int_{E}J\Phi_t(x)\frac{\partial}{\partial t}J\Phi_t(y)\Big(\langle \nabla_{a}G_{\delta}(\Phi_t(x),\Phi_t(y)),X(\Phi_t(x))\rangle \nonumber\\
&   \hspace{1cm} +\langle \nabla_{b}G_{\delta}(\Phi_t(x),\Phi_t(y)),X(\Phi_t(y))\rangle \Big)\,\dd x\dd y \nonumber\\
&   \hspace{0.5cm} +2\int_{E}\int_{E} \langle \frac{\partial}{\partial t}\Big(J\Phi_t(x)J\Phi_t(y)X(\Phi_t(x))\Big) , \nabla_{a}G_{\delta}(\Phi_t(x),\Phi_t(y)) \rangle\,\dd x\dd y \nonumber\\
&   \hspace{0.5cm} +2\int_{E}\int_{E}J\Phi_t(x)J\Phi_t(y)\biggl( \sum_{i,j=1}^{N}\frac{\partial^2 G_{\delta}}{\partial a_i \partial a_j}(\Phi_t(x),\Phi_t(y))X_i(\Phi_t(x))X_j(\Phi_t(x)) \nonumber\\
&   \hspace{1cm} +\sum_{i,j=1}^{N}\frac{\partial^2 G_{\delta}}{\partial a_i \partial b_j}(\Phi_t(x),\Phi_t(y))X_i(\Phi_t(x))X_j(\Phi_t(y)) \biggr)\,\dd x\dd y\,.
\end{align}
Using identity \eqref{eq:idnabla} and integrating by parts, we can rewrite this expression as
\begin{align*}
\frac{\partial^2}{\partial t^2} &\nl_{\delta}(E_{t})
= \int_E\int_E f(t,x,y)G_\delta(\Phi_t(x),\Phi_t(y))\,\dd x\dd y \\
& + \int_E\int_E \Bigl( \langle \nabla_a G_{\delta}(\Phi_t(x),\Phi_t(y)), g_{1}(t,x,y) \rangle + \langle \nabla_b G_{\delta}(\Phi_t(x),\Phi_t(y)), g_{2}(t,x,y) \rangle \Bigr)\,\dd x\dd y \\
& + \int_E\int_{\partial E} \Bigl( \langle \nabla_a G_{\delta}(\Phi_t(x),\Phi_t(y)), h_{1}(t,x,y) \rangle + \langle \nabla_b G_{\delta}(\Phi_t(x),\Phi_t(y)), h_{2}(t,x,y) \rangle \Bigr) \,\dd\hn(x)\dd y\,,
\end{align*}
for some functions $f,g_1,g_2,h_1,h_2$ uniformly bounded in $(-t_0,t_0)\times\overline{E}\times\overline{E}$.
It is then easily seen that
$$
\frac{\partial^2}{\partial t^2}\nl_{\delta}(E_{t})\stackrel{\delta\rightarrow0}{\longrightarrow}K(t)\;\;\mbox{uniformly in}\;t\in(-t_0,t_0)\,,
$$
where $K(t)$ is simply obtained by replacing $G_\delta$ by $G$ in the previous expression.

We finally compute \eqref{var2}.
Setting $Z:=\frac{\partial^2 \Phi}{\partial t^2}_{|_{t=0}}$ we have that
$$
\frac{\partial^2 J\Phi_t}{\partial t^2}_{|_{t=0}} = \div Z + (\div X)^2-\sum_{i,j=1}^{N}\frac{\partial X_i}{\partial x_j}\frac{\partial X_j}{\partial x_i} = \div\bigl((\div X) X\bigr)\,.
$$
Therefore, computing \eqref{var3} at $t=0$, from this identity and recalling \eqref{eq:jac} we obtain
\begin{align*}
\frac{\partial^2}{\partial t^2}{\nl_{\delta}(E_{t})}_{|_{t=0}}
&= 2 \int_{E}\int_{E} \biggl[ \div\bigl((\div X)X\bigr)(x) G_{\delta}(x,y) + \div X(x)\div X(y) G_{\delta}(x,y) \biggr]\,\dd x\dd y \\
& \hspace{.5cm} + 4 \int_E\int_E \div X(y) \sum_{i=1}^{N}\Bigl( \frac{\partial G_{\delta}}{\partial x_i}(x,y)X_i(x) + \frac{\partial G_{\delta}}{\partial y_i}(x,y)X_i(y) \Bigr) \,\dd x\dd y \\
& \hspace{.5cm} + 2 \int_E\int_E \sum_{i,j=1}^{N} \biggl( \frac{\partial G_{\delta}}{\partial x_i}(x,y)\frac{\partial X_i}{\partial x_j}(x)X_j(x) + \frac{\partial^2 G_{\delta}}{\partial x_i \partial x_j}(x,y)X_i(x)X_j(x) \\
& \hspace{1cm} + \frac{\partial^2 G_{\delta}}{\partial x_i \partial y_j}(x,y)X_i(x)X_j(y) \biggr) \,\dd x\dd y
 =: I_1 + I_2 + I_3\,.
\end{align*}
By integrating by parts in $I_1$, the sum of the first two integrals is equal to
\begin{align*}
I_1+I_2
&=
2\int_{E}\int_{E} \langle\nabla_x G_{\delta}(x,y), X(x)\rangle \bigl(\div X(x)+\div X(y)\bigr) \,\dd x\dd y \\
& \hspace{0.5cm} +2\int_{E}\int_{\partial E}G_{\delta}(x,y) \bigl(\div X(x) + \div X(y) \bigr)\langle X(x), \nu (x) \rangle\,\dd\hn(x)\dd y\,.
\end{align*}
Hence
\begin{align*}
\frac{\partial^2}{\partial t^2}&{\nl_{\delta}(E_{t})}_{|_{t=0}}
= 2\int_{E}\int_{\partial E}G_{\delta}(x,y) \bigl(\div X(x)+\div X(y)\bigl) \langle X(x), \nu(x) \rangle\,\dd\hn(x)\dd y \\
& \hspace{.5cm}+2\int_{E}\int_{E} \Bigl( \div_x \bigl( \langle \nabla_x G_{\delta}(x,y), X(x)\rangle X(x) \bigr)
+ \div_y \bigl( \langle \nabla_x G_{\delta}(x,y), X(x)\rangle X(y) \bigr) \Bigr)\,\dd x\dd y \\
& = 2\int_{E}\left( \int_{\partial E} \div_x\bigl(G_{\delta}(x,y)X(x)\bigr)\langle X(x),\nu(x) \rangle \,\dd\hn(x) \right)\,\dd y \\
& \hspace{.5cm} + 2\int_{E}\left( \int_{\partial E} \div_x\bigl(G_{\delta}(x,y)X(x)\bigr)\langle X(y),\nu(y) \rangle \,\dd\hn(y) \right)\,\dd x  \\
& = 2\int_{\partial E}\left( \int_{E}\div_x\bigl(G_{\delta}(x,y)X(x)\bigr)\,\dd y \right)\langle X(x),\nu(x) \rangle \,\dd\hn(x) \\
& \hspace{.5cm} + 2\int_{\partial E}\int_{\partial E} G_{\delta}(x,y) \langle X(x),\nu(x) \rangle\langle X(y),\nu(y) \rangle \,\dd\hn(x)\dd\hn(y)\,,
\end{align*}
where the second equality follows after having applied the divergence theorem, and the last one by Fubini's Theorem and the divergence theorem.
Thus, using the Lebesgue's Dominated Convergence Theorem to compute the limit of the previous quantity as $\delta\rightarrow0$, and recalling that $\alpha\in(0,N-1)$, we obtain
\begin{equation} \label{var4}
\begin{split}
\frac{\partial^2}{\partial t^2}{\nl_{\delta}(E_{t})}_{|_{t=0}} & = 2\int_{\partial E}\left( \int_{E}\div_x\bigl(G(x,y)X(x)\bigr)\,\dd y \right)\langle X(x),\nu(x) \rangle \,\dd\hn(x) \\
& \quad+ 2\int_{\partial E}\int_{\partial E}G(x,y) \langle X(x),\nu(x) \rangle\langle X(y),\nu(y) \rangle \,\dd\hn(x)\dd\hn(y)\,.
\end{split}
\end{equation}
We can rewrite the first integral in the previous expression as
\begin{align*}
2\int_{\partial E}\biggl( \int_{E}\div_x&\bigl(G(x,y)X(x)\bigr)\,\dd y \biggr)\langle X(x),\nu(x) \rangle \,\dd\hn(x)
 = 2\int_{\partial E}\div\bigl(v_{E}X\bigr)\langle X,\nu \rangle \,\dd\hn \\
& = 2\int_{\partial E} \Bigl( v_E(\div X)\langle X,\nu\rangle
    + \langle\nabla v_E,X_\tau\rangle\langle X,\nu\rangle
    + \partial_\nu v_E\,\langle X,\nu\rangle^2 \Bigr)\,\dd\hn \\
& = 2\int_{\partial E} \Bigl( v_E(\div X)\langle X,\nu\rangle
    - v_E\,\div_\tau \bigl(X_\tau\langle X,\nu\rangle\bigr)
    + \partial_\nu v_E\,\langle X,\nu\rangle^2 \Bigr)\,\dd\hn\,.
\end{align*}
Finally, combining this expression with \eqref{var4} and \eqref{eq:secvarper},
we obtain the formula in the statement.
\end{proof}

\subsection{Computation of $\mathcal{I}^{N,\alpha}$}
Here we want to get an explicit expression of the integral
$$
\mathcal{I}^{N,\alpha}:= \int_{B_{1}} \frac{\langle x-y, x\rangle}{|x-y|^{\alpha+2}} \,\dd y
$$
appearing in Section~\ref{sect:ball}, at least in the case $N=3$.
First of all, since $\mathcal{I}^{N,\alpha}$ is independent of $x\in S^{N-1}$, we fix $x=e_{1}$.
By Fubini's Theorem we get
$$
\mathcal{I}^{N,\alpha}
=
\int_{B_{1}} \frac{1-y_{1}}{|e_{1}-y|^{\alpha+2}} \,\dd y
=
\int_{-1}^{1} \biggl( \int_{B_t} \frac{1-t}{\bigl((1-t)^2+|z|^2\bigr)^{\frac{\alpha+2}{2}}} \,\dd \mathcal{L}^{N-1}(z) \biggr) \dd t\,,
$$
where $B_t:=B^{N-1}(0, \sqrt{1-t^2})$ denotes a $(N-1)$-dimensional ball of radius $\sqrt{1-t^2}$ centered at the origin.
To treat the inner integral, we apply the co-area formula (see \cite[equation~(2.74)]{AFP}),
by integrating on the level sets of the function $f_t(z):=\sqrt{(1-t)^2+|z|^2}$, $z\in\R^{N-1}$:
setting $\delta(r)=\sqrt{r^2-(1-t)^2}$, we get
\begin{align*}
\int_{B_t} \frac{1}{\bigl((1-t)^2+|z|^2\bigr)^{\frac{\alpha+2}{2}}} \,\dd \mathcal{L}^{N-1}(z)
& =
\int_{1-t}^{\sqrt{2(1-t)}} \biggl( \int_{\partial B^{N-1}(0,\delta(r))} \frac{1}{r^{\alpha+1}\sqrt{r^2-(1-t)^2}}\,\dd\mathcal{H}^{N-2} \biggr) \,\dd r \\
& =
(N-1)\omega_{N-1} \int_{1-t}^{\sqrt{2(1-t)}} \frac{(r^2-(1-t)^2)^\frac{N-3}{2}}{r^{\alpha+1}}\,\dd r.
\end{align*}
Therefore
$$
\mathcal{I}^{N,\alpha}
=
(N-1)\omega_{N-1} \int_{-1}^{1}(1-t)
\biggl( \int_{1-t}^{\sqrt{2(1-t)}} \frac{(r^2-(1-t)^2)^\frac{N-3}{2}}{r^{\alpha+1}}\,\dd r \biggr) \, \dd t.
$$
From real analysis we know that we can write the inner integral in term of simple functions if and only if $N$ is odd or $\alpha$ is an integer. Since we are interested in the physical case ($N=3$, $\alpha=1$), we just compute the above integral for $N=3$, obtaining
\begin{equation}\label{eq:integral}
\mathcal{I}^{3,\alpha} = 2\pi \frac{2^{2-\alpha}}{(4-\alpha)(2-\alpha)}\,.
\end{equation}


\bigskip
\bigskip
\noindent
{\bf Acknowledgments.}
{We wish to thank Massimiliano Morini for having introduced us to the study of this problem and for multiple helpful discussions.}


\bibliographystyle{siam} 
\bibliography{bibliografia}

\def\cprime{$'$}
\begin{thebibliography}{10}

\bibitem{AceFusMor}
{\sc E.~Acerbi, N.~Fusco, and M.~Morini}, {\em Minimality via second variation
  for a nonlocal isoperimetric problem}, Comm. Math. Phys., 322 (2013),
  pp.~515--557.

\bibitem{Alm}
{\sc F.~J. Almgren, Jr.}, {\em Existence and regularity almost everywhere of
  solutions to elliptic variational problems among surfaces of varying
  topological type and singularity structure}, Ann. of Math. (2), 87 (1968),
  pp.~321--391.

\bibitem{AFP}
{\sc L.~Ambrosio, N.~Fusco, and D.~Pallara}, {\em Functions of bounded
  variation and free discontinuity problems}, Oxford Mathematical Monographs,
  The Clarendon Press Oxford University Press, New York, 2000.

\bibitem{BonMor}
{\sc M.~Bonacini and M.~Morini}, {\em An ${L}^1$ local minimality criterion for
  the {M}umford-{S}hah functional}.
\newblock In preparation.

\bibitem{CagMorMor}
{\sc F.~Cagnetti, M.~G. Mora, and M.~Morini}, {\em A second order minimality
  condition for the {M}umford-{S}hah functional}, Calc. Var. Partial
  Differential Equations, 33 (2008), pp.~37--74.

\bibitem{CapJulPis}
{\sc G.~M. Capriani, V.~Julin, and G.~Pisante}, {\em A {Q}uantitative {S}econd
  {O}rder {M}inimality {C}riterion for {C}avities in {E}lastic {B}odies}, SIAM
  J. Math. Anal., 45 (2013), pp.~1952--1991.

\bibitem{ChoPel1}
{\sc R.~Choksi and M.~A. Peletier}, {\em Small volume fraction limit of the
  diblock copolymer problem: {I}. {S}harp-interface functional}, SIAM J. Math.
  Anal., 42 (2010), pp.~1334--1370.

\bibitem{ChoPel2}
\leavevmode\vrule height 2pt depth -1.6pt width 23pt, {\em Small
  volume-fraction limit of the diblock copolymer problem: {II}.
  {D}iffuse-interface functional}, SIAM J. Math. Anal., 43 (2011),
  pp.~739--763.

\bibitem{ChoSte}
{\sc R.~Choksi and P.~Sternberg}, {\em On the first and second variations of a
  nonlocal isoperimetric problem}, J. Reine Angew. Math., 611 (2007),
  pp.~75--108.

\bibitem{CicLeo}
{\sc M.~Cicalese and G.~P. Leonardi}, {\em A {S}election {P}rinciple for the
  {S}harp {Q}uantitative {I}soperimetric {I}nequality}, Arch. Ration. Mech.
  Anal., 206 (2012), pp.~617--643.

\bibitem{CicSpa}
{\sc M.~Cicalese and E.~Spadaro}, {\em Droplet minimizers of an isoperimetric
  problem with long-range interactions}, Comm. Pure Appl. Math., 66 (2013),
  pp.~1298--1333.

\bibitem{PalVal}
{\sc E.~Di~Nezza, G.~Palatucci, and E.~Valdinoci}, {\em Hitchhiker's guide to
  the fractional {S}obolev spaces}, Bull. Sci. Math., 136 (2012), pp.~521--573.

\bibitem{EspFus}
{\sc L.~Esposito and N.~Fusco}, {\em A remark on a free interface problem with
  volume constraint}, J. Convex Anal., 18 (2011), pp.~417--426.

\bibitem{FusMor}
{\sc N.~Fusco and M.~Morini}, {\em Equilibrium configurations of epitaxially
  strained elastic films: second order minimality conditions and qualitative
  properties of solutions}, Arch. Ration. Mech. Anal., 203 (2012),
  pp.~247--327.

\bibitem{Gamow}
{\sc G.~Gamow}, {\em Mass defect curve and nuclear constitution}, Proc. R. Soc.
  Lond. Ser. A, 126 (1930), pp.~632--644.

\bibitem{GT}
{\sc D.~Gilbarg and N.~S. Trudinger}, {\em Elliptic partial differential
  equations of second order}, Classics in Mathematics, Springer-Verlag, Berlin,
  2001.
\newblock Reprint of the 1998 edition.

\bibitem{GolMurSerI}
{\sc D.~Goldman, C.~Muratov, and S.~Serfaty}, {\em The ${\Gamma}$-limit of the
  two-dimensional {O}hta-{K}awasaki energy. {I}. {D}roplet density}.
\newblock Preprint (2012).

\bibitem{GolMurSerII}
\leavevmode\vrule height 2pt depth -1.6pt width 23pt, {\em The {$\Gamma$}-limit
  of the two-dimensional {O}hta-{K}awasaki energy. {II}. {D}roplet arrangement
  at the sharp interface level via the renormalized energy}.
\newblock Preprint (2012).

\bibitem{Gro}
{\sc H.~Groemer}, {\em Geometric applications of {F}ourier series and spherical
  harmonics}, vol.~61 of Encyclopedia of Mathematics and its Applications,
  Cambridge University Press, Cambridge, 1996.

\bibitem{Jul}
{\sc V.~Julin}, {\em Isoperimetric problem with a {C}oulombic repulsive term},
  Indiana Univ. Math. J.
\newblock (accepted paper).

\bibitem{JulPis}
{\sc V.~Julin and G.~Pisante}, {\em Minimality via second variation for
  microphase separation of diblock copolymer melts}.
\newblock Preprint (2013).

\bibitem{KnuMur1}
{\sc H.~Kn\"{u}pfer and C.~B. Muratov}, {\em On an isoperimetric problem with a
  competing non-local term. {I}. {T}he planar case}, Comm. Pure Appl. Math., 66
  (2013), pp.~1129--1162.

\bibitem{KnuMur2}
\leavevmode\vrule height 2pt depth -1.6pt width 23pt, {\em On an isoperimetric
  problem with a competing non-local term. {II}. {T}he general case}, to appear
  on Comm. Pure Appl. Math.,  (2013).

\bibitem{LiebLoss}
{\sc E.~H. Lieb and M.~Loss}, {\em Analysis}, vol.~14 of Graduate Studies in
  Mathematics, American Mathematical Society, Providence, RI, second~ed., 2001.

\bibitem{LuOtt}
{\sc J.~Lu and F.~Otto}, {\em Nonexistence of minimizers for
  {T}homas-{F}ermi-{D}irac-von {W}eizs\"{a}cker model}, Comm. Pure Appl. Math.
\newblock (accepted paper).

\bibitem{Mag}
{\sc F.~Maggi}, {\em Sets of finite perimeter and geometric variational
  problems}, vol.~135 of Cambridge Studies in Advanced Mathematics, Cambridge
  University Press, Cambridge, 2012.

\bibitem{Mur}
{\sc C.~B. Muratov}, {\em Theory of domain patterns in systems with long-range
  interactions of {C}oulomb type}, Phys. Rev. E (3), 66 (2002), pp.~1--25.

\bibitem{Mur2}
\leavevmode\vrule height 2pt depth -1.6pt width 23pt, {\em Droplet phases in
  non-local {G}inzburg-{L}andau models with {C}oulomb repulsion in two
  dimensions}, Comm. Math. Phys., 299 (2010), pp.~45--87.

\bibitem{OhtaKaw}
{\sc T.~Ohta and K.~Kawasaki}, {\em Equilibrium morphology of block copolymer
  melts}, Macromolecules, 19 (1986), pp.~2621--2632.

\bibitem{Rig}
{\sc S.~Rigot}, {\em Ensembles quasi-minimaux avec contrainte de volume et
  rectifiabilit\'e uniforme}, M\'em. Soc. Math. Fr. (N.S.),  (2000).

\bibitem{SchSim}
{\sc R.~Schoen and L.~Simon}, {\em A new proof of the regularity theorem for
  rectifiable currents which minimize parametric elliptic functionals}, Indiana
  Univ. Math. J., 31 (1982), pp.~415--434.

\bibitem{Sim}
{\sc L.~Simon}, {\em Lectures on geometric measure theory}, vol.~3 of
  Proceedings of the Centre for Mathematical Analysis, Australian National
  University, Australian National University Centre for Mathematical Analysis,
  Canberra, 1983.

\bibitem{SteTop}
{\sc P.~Sternberg and I.~Topaloglu}, {\em On the global minimizers of a
  nonlocal isoperimetric problem in two dimensions}, Interfaces Free Bound., 13
  (2011), pp.~155--169.

\bibitem{Tam}
{\sc I.~Tamanini}, {\em Boundaries of {C}accioppoli sets with
  {H}\"older-continuous normal vector}, J. Reine Angew. Math., 334 (1982),
  pp.~27--39.

\bibitem{Whi}
{\sc B.~White}, {\em A strong minimax property of nondegenerate minimal
  submanifolds}, J. Reine Angew. Math., 457 (1994), pp.~203--218.

\end{thebibliography}

\end{document}